\documentclass[a4paper,12pt]{scrartcl}
\usepackage[english]{babel}

\usepackage{microtype}
\usepackage{fontspec}
\usepackage{amssymb}
\usepackage[bold-style=ISO]{unicode-math}
\usepackage{mathtools}
\usepackage{tabularray}
\UseTblrLibrary{amsmath}
\UseTblrLibrary{booktabs}
\usepackage[amsmath,thmmarks,hyperref]{ntheorem}
\usepackage{csquotes}
\usepackage[shortlabels]{enumitem}
\usepackage{needspace}
\usepackage{siunitx}

\usepackage{hyperref}

\makeatletter\newcommand\customitem[1]{\item[#1{\MakeLinkTarget[item]{}}]\def\@currentlabel{#1}}\makeatother

\usepackage[
	backend=biber,
	style=numeric,
	datamodel=reviewids,
	giveninits = true,
	sorting = nyt,
	maxnames = 8,
	sortlocale = en_US,
	doi=false
]{biblatex}

\DeclareFieldFormat{zblnumber}{%
  \textsc{Zbl}\addcolon\space
  \href{https://zbmath.org/#1}{\nolinkurl{#1}}%
}

\DeclareFieldFormat{jfmnumber}{%
  \textsc{JFM}\addcolon\space
  \href{https://zbmath.org/#1}{\nolinkurl{#1}}%
}

\DeclareFieldFormat{mrnumber}{%
  \textsc{MR}\addcolon\space
  \href{https://mathscinet.ams.org/mathscinet/article?mr=#1}{\nolinkurl{#1}}%
}

\newbibmacro*{mathlinks}{%
  \printfield{zblnumber}%
  \newunit\newblock
  \printfield{jfmnumber}%
  \newunit\newblock
  \printfield{mrnumber}%
}

\renewbibmacro*{doi+eprint+url}{%
  \iftoggle{bbx:doi}
    {\printfield{doi}}
    {}%
  \newunit\newblock
  \usebibmacro{mathlinks}%
  \newunit\newblock
  \iftoggle{bbx:eprint}
    {\usebibmacro{eprint}}
    {}%
  \newunit\newblock
  \iftoggle{bbx:url}
    {\usebibmacro{url+urldate}}
    {}%
}

\newbibmacro{string+doi}[1]{%
	 \iffieldundef{doi}{#1}{\href{http://dx.doi.org/\thefield{doi}}{#1}}}
	 \DeclareFieldFormat{title}{\usebibmacro{string+doi}{\mkbibemph{#1}}}
	 \DeclareFieldFormat[article]{title}{\usebibmacro{string+doi}{\mkbibquote{#1}}}

\bibliography{mendelsohn}

\usepackage{tikz}

\AtBeginDocument{\renewcommand{\setminus}{\mathop{\backslash}}}

\addtokomafont{disposition}{\rmfamily}

\newcommand{\thmheadernewline}{\Needspace*{3\baselineskip}\leavevmode\par\nobreak\noindent\ignorespaces}

\newcommand{\F}{\mathbb{F}}

\newcommand{\C}{\mathbb{C}}
\newcommand{\N}{\mathbb{N}}
\newcommand{\Z}{\mathbb{Z}}

\newcommand{\vek}[1]{\symbf{#1}}

\newcommand{\qbinom}[3]{\genfrac{[}{]}{0pt}{}{#1}{#2}_{#3}}

\theorembodyfont{\upshape}
\theoremseparator{.}
\theoremsymbol{\ensuremath{\diamond}}
\newtheorem{theorem}{Theorem}
\newtheorem{corollary}[theorem]{Corollary}
\newtheorem{lemma}{Lemma}[section]
\newtheorem{fact}[lemma]{Fact}

\newtheorem{definition}[lemma]{Definition}

\newtheorem{remark}[lemma]{Remark}
\newtheorem{example}[lemma]{Example}

\theoremheaderfont{\itshape}
\theorembodyfont{\upshape}
\theoremstyle{nonumberplain}
\theoremseparator{.}
\theoremsymbol{\rule{1ex}{1ex}}
\newtheorem{proof}{Proof}

\numberwithin{table}{section}

\DeclareMathOperator{\OA}{OA}
\DeclareMathOperator{\PG}{PG}

\DeclareMathOperator{\rk}{rk}

\providecommand{\norm}[1]{\lVert#1\rVert}

\title{Intersection numbers\\for designs in regular semilattices}
\author{Michael Kiermaier\thanks{
University of Bayreuth, Institute for Mathematics, 95440 Bayreuth, Germany
\newline
email:~\texttt{michael.kiermaier@uni-bayreuth.de}
\newline
homepage:~\url{https://mathe2.uni-bayreuth.de/michaelk/}}
\and
Lukas Klawuhn\thanks{
Paderborn University, Department of Mathematics, 33098 Paderborn, Germany
\newline
email:~\texttt{klawuhn@math.upb.de}}
}
\date{August 14, 2026}

\begin{document}

\maketitle

% Prevent the abstract from spilling over to page 2 when using XeLaTeX.
% This would not be required with LuaLaTeX, which is not available on arXiv.
\vspace{-6mm}

\begin{abstract}
	We generalize intersection numbers for combinatorial designs to designs in finite meet-semilattices satisfying suitable regularity conditions.
	While designs in regular semilattices go back to Delsarte, our regularity assumptions are weaker than his and need not give rise to an association scheme.
	In this framework, we extend Mendelsohn's equations, prove a generalized Singleton bound with Steiner systems as equality cases, and determine the block intersection distribution at any block of a Steiner system.
	In particular, this distribution is independent of the chosen block.

	Specializing to several classical semilattice families, our results recover a number of well-known distributions in coding and design theory.
	In the Hamming and the $q$-Hamming (or bilinear forms) schemes, they give the local distance distributions of MDS and MRD codes, respectively.
	In the Johnson and $q$-Johnson (or Graßmann) schemes, they reproduce the block intersection distribution of classical and $q$-analog Steiner systems, equivalently the distance distribution of diameter-perfect constant-weight codes and diameter-perfect constant-dimension subspace codes.
	For the $q$-Johnson schemes, to the best of our knowledge, this result is new.
	As a further illustration, we apply our theory to designs of perfect matchings.

	Our approach provides a unified treatment of these cases in the strongest form known in the literature, determining the distribution relative to each individual block or codeword, without averaging and without linearity or additivity assumptions.
	Moreover, it identifies the natural double-counting objects underlying these distributions, leading to formulas in the regularity parameters of the semilattice and avoiding the more cumbersome expressions that arise in eigenvalue-based approaches via the ambient association scheme.
\end{abstract}

\section{Introduction}
\subsection{Historical overview}
In parallel developments, Mendelsohn and Goethals (1969--1971) and, independently, Oberschelp (1972) derived a system of linear equations for the intersection numbers of a combinatorial design \cite{Mendelsohn-1969-NAMS16[6]:984,Mendelsohn-1971,Goethals-1970-InstStatistMimeoSer600_29,Oberschelp-1972-MPSemBNF19:55-67}, now commonly referred to as the \emph{Mendelsohn equations}.
An alternative formulation, parametrizing the lower intersection numbers by the higher ones, was later obtained by Köhler~\cite{Koehler-1989-DM73[1-2]:133-142} and, independently, by Harnau~\cite[Satz~1]{Harnau-1988-RostMKoll34:47-52}.
The Mendelsohn equations impose strong restrictions on the possible intersection structure of a design and have been used to establish the non-existence of designs for certain admissible parameter sets.
As an example, we mention the admissible parameter set $5$-$(19,9,7)$, whose non-existence was proved by Harnau in 1988~\cite{Harnau-1988-RostMKoll34:47-52}.
Besides yielding non-existence results, the combinatorial information encoded in the Mendelsohn equations is also useful for narrowing the search space for designs with prescribed parameters.

In~1976, Delsarte introduced a regular semilattice structure underlying certain association schemes, within which the combinatorial notion of a design extends naturally, thus placing the classical theory in a substantially broader setting \cite{Delsarte-1976-JCTSA20[2]:230-243}.
His article included four families of regular semilattices that may now be regarded as classical:
the \emph{Johnson} and \emph{Hamming} semilattices and their $q$-analogs, namely the \emph{$q$-Johnson} (or \emph{Graßmann}) and \emph{$q$-Hamming} (or \emph{bilinear forms}) semilattices.
For the first three families, the resulting designs specialize to classical block designs, subspace designs, and orthogonal arrays, all of which have been studied extensively as independent combinatorial objects.
In the Hamming and $q$-Hamming semilattices, Steiner systems, that is, designs of index one, correspond to MDS and MRD codes, respectively, which are themselves objects of major importance in coding theory.
The local distance distribution of these codes is uniquely determined by their parameters.
The corresponding formulas have been discovered repeatedly and independently in the literature for both families, based on several approaches:
One may employ the theory of dual codes and the MacWilliams equations, although this is, at least in its basic form, restricted to linear codes.
Alternatively, one may use the theory of association schemes, which removes this restriction but, in the unrestricted setting, yields only the inner distribution of the code, an essentially averaged and therefore weaker version of the local distance distribution, and involves rather intricate expressions arising from the eigenvalues of the corresponding scheme.
The most direct approach is combinatorial rather than algebraic and based on a double counting argument, which, to the best of our knowledge, is found in the literature only for MDS codes.

\subsection{New results and outline of the article}

In this article, we investigate finite regular posets, introduced in Section~\ref{sec:regular posets}, whose regularity requirements are somewhat less restrictive than those studied by Delsarte and which therefore do not necessarily support the structure of an association scheme.
Section~\ref{sec:classical_lattices} then presents the four classical families of regular semilattices in algebraic combinatorics, namely the Hamming lattice, the Johnson lattice, and their $q$-analogs.
Section~\ref{sec:poset_design} provides the necessary preliminaries on designs in regular posets, including the divisibility conditions of Theorem~\ref{thm:divisibility}, which, to the best of our knowledge, have not previously appeared in this level of generality.
An overview of designs in the four classical families is given in Example~\ref{ex:designs_classical_semilattices}.

Section~\ref{sec:semilattice_mendelsohn} adds a $\wedge$-semilattice structure to the regularity assumptions.
In Theorem~\ref{thm:general:mendelsohn}, based on a double-counting argument, we establish the generalized Mendelsohn equations, and we derive their alternative Köhler form.
For the $q$-Hamming semilattices, and partially also for the ordinary Hamming semilattices, this result appears to be new and is stated as Corollary~\ref{cor:hamming:mendelsohn}.
In Section~\ref{sec:steiner}, we introduce codes in regular semilattices.
We establish an extended Singleton bound in Theorem~\ref{thm:singleton}, including a characterization of the equality cases as Steiner systems.
Based on Theorem~\ref{thm:general:mendelsohn}\ref{thm:general:mendelsohn:koehler}, we derive in Theorem~\ref{thm:general:steiner_block_intersection_distribution} a formula for the block intersection distribution of Steiner systems, showing in particular that it is independent of the choice of the base block.
For the $q$-Johnson semilattices, this result has not been published previously and is stated as Corollary~\ref{cor:johnson_steiner_intersection_distribution}.
Specializing to the Hamming and $q$-Hamming semilattices recovers the local distance distributions of MDS and MRD codes.
Both of these results have been rediscovered several times independently in the literature, sometimes in weaker forms and sometimes via alternative, more involved proofs.
A detailed historical overview is provided in Example~\ref{ex:codes_hamming} and Example~\ref{ex:codes_q_hamming}.

Finally, in Section~\ref{sec:matchings}, we apply our theory to the semilattice of matchings of the complete graph, whose top level consists of the perfect matchings.
Since the corresponding semilattice is not Delsarte-regular, this example illustrates that the additional level of generality considered here extends beyond Delsarte's framework and leads to new results, including the Mendelsohn equations in Corollary~\ref{cor:matchings:mendelsohn} and the formula for the block intersection distribution of Steiner systems of perfect matchings in Corollary~\ref{cor:matchings:steiner_block_intersection_distribution}.

\section{Regular posets}\label{sec:regular posets}

Our objective is to extend intersection numbers of designs, the associated Mendelsohn equations, and results derived from them to a substantially broader class of designs, namely designs in regular semilattices, as introduced by Delsarte in~\cite{Delsarte-1976-JCTSA20[2]:230-243}.
This section provides the necessary preliminaries.
These require only a finite poset with certain regularity properties and do not yet assume any semilattice structure.
In fact, our setting is more general than that of~\cite{Delsarte-1976-JCTSA20[2]:230-243},%
\footnote{
        Compared to~\cite{Delsarte-1976-JCTSA20[2]:230-243}, we replace the regularity property~\ref{regular:pi} by the weaker property~\ref{regular:theta}; moreover, most statements in this section remain valid even without~\ref{regular:nu}, see Remark~\ref{rem:general_without_nu}.
        In addition, we do not require $(X,{\leq})$ to be graded, and we allow complex-valued functions~$\beta$ (although the extension from real-valued functions is straightforward).
}
and is also slightly more general than the one considered in~\cite{Suda-2012-DM312[10]:1827-1831}.
We note, however, that even this level of generality does not fully capture all design-like structures studied in the literature.
Examples include designs in finite polar spaces~\cite{Kiermaier-Schmidt-Wassermann-2025-DCC93[4]:1143-1162} and designs of perfect matchings~\cite{Bamberg-Klawuhn-2026-AlgebrComb9[3]:789-809}.
Other general frameworks for designs and/or generalizations of regular semilattices have been studied in~\cite{Martin-2001-DIMACS56:223-239}, \cite[Sec.~4]{Guo-Ma-Wang-2013-SChinaM56[11]:2393-2407}, and~\cite{Zhang-Liu-Li-2016-DiscreteMathAlgorithmsAppl8[1]:P1650005}.
For the fundamental theory of posets, the reader is referred to~\cite[Ch.~3]{Stanley-2012-EnumerativeCombinatoricsI-2nd}.

Let $(X,{\leq})$ be a finite poset with bottom element $\bot$.
For $a,b\in X$, we write $a \lessdot b$ if $b$ covers $a$, i.\,e., if $a < b$ and $a \leq x \leq b$ implies $x\in\{a,b\}$.
As usual, the \emph{length} $\ell(a,b)$ of elements $a,b\in X$ with $a \leq b$ is defined as the length $r$ of the longest chain $a = x_0 < \ldots < x_r = b$ from $a$ to $b$.%
\footnote{
	Such a chain necessarily has the form $a = x_0 \lessdot \ldots \lessdot x_r = b$.
}
Moreover, the \emph{height} of an element $a\in X$ is $h(a) \coloneqq \ell(\bot,a)$.
By concatenation of chains, for all $a,b,c\in X$ with $a \leq b \leq c$ we have $\ell(a,c) \geq \ell(a,b) + \ell(b,c)$ and, as a special case, $h(b) \geq h(a) + \ell(a,b)$.
We extend the notion of the height to non-empty subsets $A \subseteq X$ via $h(A) = \max_{x\in A} h(x)$.

The height function partitions $X$ into the \emph{levels}
\[
	X_r = h^{-1}(\{r\}) = \{x\in X \mid h(x) = r\}
\]
where $r\in\{0,\ldots,n\}$ and $n = h(X)$ denotes the largest height in $X$.
Clearly, $X_0 = \{\bot\}$, and each level $X_r$ is an antichain.
In what follows, the top level $\Omega = X_n$ will play a special role.

\begin{lemma}\label{lem:chain_Xr}
	Let $(X,{\leq})$ be a finite poset with bottom element $\bot$.
	Let $a \in X$ and let $r = h(a)$.
	Then there exists a chain
	\[
		\bot = x_0 < \ldots < x_r = a
	\]
	of length $r$.
	Moreover, for each such chain, $h(x_s) = s$ for all $s\in\{0,\ldots,r\}$.
\end{lemma}

\begin{proof}
	The existence follows from $h(a) = r$ and the definition of the height.

	Let $\bot = x_0 < \ldots < x_r$ be a chain of length $r$ with $h(x_r) = r$, and let $s\in \{0,\ldots,r\}$.
	The subchain $\bot = x_0 < \ldots < x_s$ has length $s$, hence $h(x_s) \geq s$.

	Similarly, the subchain $x_s < \ldots < x_r$ has length $r-s$, hence $\ell(x_s, x_r) \geq r-s$.
	Now the inequality $h(x_r) \geq h(x_s) + \ell(x_s, x_r)$ yields
	\[
		h(x_s) \leq h(x_r) - \ell(x_s, x_r) \leq r - (r-s) = s\text{,}
	\]
	so, altogether, $h(x_s) = s$.
\end{proof}

\begin{lemma}\label{lem:levels_nonempty}
	Let $(X,{\leq})$ be a finite poset with bottom element $\bot$, and let $n = h(X)$.
	Then the levels $X_0, \ldots, X_n$ are non-empty.
\end{lemma}

\begin{proof}
	Let $s\in\{0,\ldots,n\}$.
	Since $n$ is the largest height, there exists an element $a\in X_n$.
	By Lemma~\ref{lem:chain_Xr}, there exists a chain $\bot = x_0 < \ldots < x_n = a$ of length $n$ with $h(x_s) = s$.
	Thus $x_s\in X_s$, and therefore, $X_s\neq\emptyset$.
\end{proof}

\begin{definition}
	Let $(X,{\leq})$ be a finite poset with bottom element $\bot$.
	We define the following regularity properties.
	\begin{enumerate}[(a)]
		\customitem{(R$\theta$)}\label{regular:theta}
		For all $r\in\{0,\ldots,n\}$, the number
		\[
			\theta(r) = \#\{x\in \Omega \mid a \leq x\}
		\]
		is constant over all choices of $a\in X_r$.
		\customitem{(R$\mu$)}\label{regular:mu}
		For all $r,s\in\{0,\ldots,n\}$, the number
		\[
			\mu(r,s) = \#\{x\in X_s \mid a \leq x \leq b\}
		\]
		is constant over all choices of $(a,b)\in X_r \times \Omega$ with $a \leq b$.%
		\footnote{
			Such admissible pairs exist:
			choose $b\in\Omega$, which is non-empty since $\Omega=X_n$ and $n=h(X)$; applying Lemma~\ref{lem:chain_Xr} to $b$ yields an element $a\in X_r$ with $a\leq b$.
			Hence $\mu(r,s)$ is well-defined.
		}
		\customitem{(R$\nu$)}\label{regular:nu}
		For all $r,s\in\{0,\ldots,n\}$, the number
		\[
			\nu(r,s) = \#\{x\in X_r \mid x \leq a\}
		\]
		is constant over all choices of $a\in X_s$.
	\end{enumerate}
	A finite poset will be called \emph{regular} if it has a bottom element and satisfies all the regularity properties~\ref{regular:theta},~\ref{regular:mu}, and~\ref{regular:nu}.

	As a mnemonic, we suggest \emph{\textbf{t}op} for $\theta$ (counting upward), \emph{\textbf{m}iddle} for $\mu$ (counting in between), and \emph{\textbf{n}adir} or \emph{\textbf{n}ether} for $\nu$ (counting downward).
\end{definition}

\begin{lemma}\label{lem:omega_cofinal}
	Let $(X,{\leq})$ be a finite poset with bottom element $\bot$ and the regularity property~\ref{regular:theta}.
	Then
	\[
		\theta(r) \geq 1
	\]
	for all $r\in \{0,\ldots,n\}$.
	In other words, the top level $\Omega$ is \emph{cofinal} in $(X,{\leq})$, meaning that for every $a\in X$, there exists a top level element $x\in\Omega$ with $a \leq x$.
\end{lemma}

\begin{proof}
	Fix $r\in\{0,\ldots,n\}$.
	Since $\Omega = X_n$ and $n = h(X)$, there exists an $x'\in\Omega$.
	Applying Lemma~\ref{lem:chain_Xr} to $x'$ yields an element $a' \in X_r$ with $a' \leq x'$.
	Hence, $\theta(r) \geq 1$.
	The cofinality statement is now a direct consequence of the regularity property~\ref{regular:theta}.
\end{proof}

For a finite regular poset $(X,{\leq})$, the regularity properties~\ref{regular:theta}, \ref{regular:nu}, and~\ref{regular:mu}, together with an existence argument analogous to the proof of Lemma~\ref{lem:omega_cofinal}, imply that for all $r,s\in\{0,\ldots,n\}$ with $r\le s$,
\[
	\nu(r,s)\ge 1
	\qquad\text{and}\qquad
	\mu(r,s)\ge 1\text{.}
\]
Moreover, the following border cases hold.
For all $r\in\{0,\ldots,n\}$,
\[
	\mu(r,r) = \mu(r,n) = \nu(0,r) = \nu(r,r) = \theta(n) = 1
	\qquad\text{and}\qquad
	\mu(0,r) = \nu(r,n)\text{,}
\]
and, in addition,
\[
	\theta(0) = \#\Omega\text{.}
\]

\begin{remark}
	A finite regular poset $(X,{\leq})$ is not necessarily graded; Figure~\ref{fig:counterex_not_graded} shows a counterexample.
\end{remark}

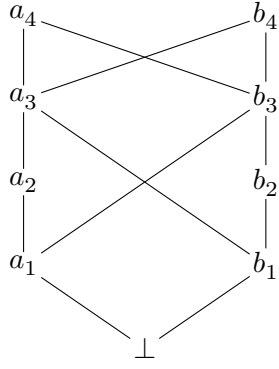
\begin{figure}[ht]
\centering
\begin{tikzpicture}[
  x=1.6cm, y=1.1cm,
  every node/.style={font=\small, inner sep=1pt},
  edge/.style={line width=0.4pt},
]

\node (bot) at (0,0) {$\bot$};

\node (a1) at (-1,1) {$a_1$};
\node (b1) at ( 1,1) {$b_1$};

\node (a2) at (-1,2) {$a_2$};
\node (b2) at ( 1,2) {$b_2$};

\node (a3) at (-1,3) {$a_3$};
\node (b3) at ( 1,3) {$b_3$};

\node (a4) at (-1,4) {$a_4$};
\node (b4) at ( 1,4) {$b_4$};

\draw[edge] (bot) -- (a1);
\draw[edge] (a1)  -- (a2);
\draw[edge] (a2)  -- (a3);
\draw[edge] (a3)  -- (a4);

\draw[edge] (bot) -- (b1);
\draw[edge] (b1)  -- (b2);
\draw[edge] (b2)  -- (b3);
\draw[edge] (b3)  -- (b4);

\draw[edge] (a1) -- (b3);
\draw[edge] (b1) -- (a3);
\draw[edge] (a3) -- (b4);
\draw[edge] (b3) -- (a4);

\end{tikzpicture}
\caption{A non-graded finite regular poset}\label{fig:counterex_not_graded}
\end{figure}

\begin{lemma}[{{\cite[Lem.~2.1]{Suda-2012-DM312[10]:1827-1831}}}]\label{lem:theta_double}
	Let $(X,{\leq})$ be a finite regular poset.
	Then, for all $r,s\in\{0,\ldots,n\}$, the number%
	\footnote{
		This extends the notation $\theta$ from one to two arguments.
		The previous one-variable definition is recovered as $\theta(r) = \theta(r,n)$.
	}
	\[
		\theta(r,s) \coloneqq \#\{x\in X_s \mid a \leq x\}
	\]
	is constant over all choices of $a\in X_r$.
	The exact value is
	\[
		\theta(r,s)
		= \frac{\theta(r)}{\theta(s)} \mu(r,s)\text{.}
	\]
\end{lemma}

\begin{proof}
	Let $a\in X_r$ and $\xi = \#\{x\in X_s \mid a \leq x\}$.
	We count the set $S$ of all $(x,y) \in X_s \times \Omega$ with $a \leq x \leq y$ in two ways.
	On the one hand, there are $\theta(r)$ choices for $y\in \Omega$ with $a\leq y$ and for each such $y$ there are $\mu(r,s)$ choices for $x\in X_s$ with $a \leq x \leq y$.
	So
	\[
		\#S = \theta(r)\mu(r,s)\text{.}
	\]
	On the other hand, there are $\xi$ choices for $x\in X_s$ with $a \leq x$, and for each such $x$ there are $\theta(s)$ choices for $y\in \Omega$ with $x \leq y$.
	Thus
	\[
		\#S = \xi\theta(s)\text{.}
	\]
	Equating both expressions for $\#S$, we get $\xi = \frac{\theta(r)}{\theta(s)}\mu(r,s)$, which is independent of the choice of $a\in X_r$.
\end{proof}

\begin{remark}\label{rem:truncation_not_regular}
	In the regularity property~\ref{regular:theta}, the top level $\Omega$ plays a distinguished role.
	Lemma~\ref{lem:theta_double} shows, however, that an analogous regularity property holds with any level $X_r$ in place of $\Omega$.
	The regularity property~\ref{regular:nu} is already formulated for arbitrary levels and does not single out the top level.
	It is therefore natural to ask whether the role of $\Omega$ in the regularity property~\ref{regular:mu} can likewise be replaced by an arbitrary level.

	However, this is not true, even when additionally imposing the requirement that the poset be graded.
	Figure~\ref{fig:counterex_truncation} provides an explicit counterexample.
	One checks that this poset is graded and meets all regularity properties.
	However, the $\mu(1,2)$-property fails if $X_3$ is taken in place of the top level:
	There are two elements $x\in\{c_1,c_2\}$ of height $2$ with $a \leq x \leq e$, but only one element of height $2$ with $a \leq x \leq f$, namely $x = c_3$.
\end{remark}

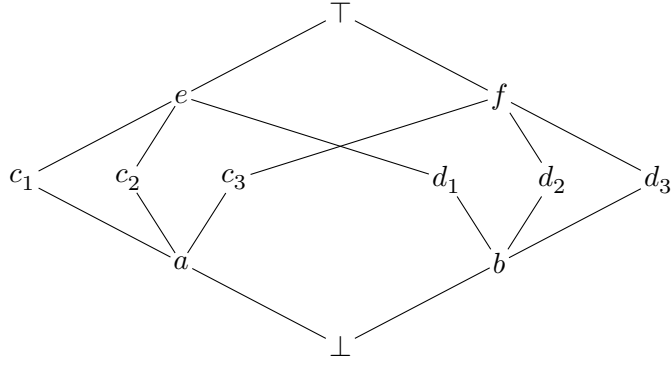
\begin{figure}[ht]
	\centering
	\begin{tikzpicture}[
		x=1.4cm, y=1.1cm,
		every node/.style={font=\small, inner sep=1pt},
		edge/.style={line width=0.4pt}
	]
		\node (bot) at (0,0) {$\bot$};

		\node (a) at (-1.5,1) {$a$};
		\node (b) at ( 1.5,1) {$b$};

		\node (c1) at (-3,2) {$c_1$};
		\node (c2) at (-2,2) {$c_2$};
		\node (c3) at (-1,2) {$c_3$};

		\node (d1) at ( 1,2) {$d_1$};
		\node (d2) at ( 2,2) {$d_2$};
		\node (d3) at ( 3,2) {$d_3$};

		\node (e) at (-1.5,3) {$e$};
		\node (f) at ( 1.5,3) {$f$};

		\node (top) at (0,4) {$\top$};

		\draw[edge] (bot) -- (a);
		\draw[edge] (bot) -- (b);

		\draw[edge] (a) -- (c1);
		\draw[edge] (a) -- (c2);
		\draw[edge] (a) -- (c3);

		\draw[edge] (b) -- (d1);
		\draw[edge] (b) -- (d2);
		\draw[edge] (b) -- (d3);

		\draw[edge] (c1) -- (e);
		\draw[edge] (c2) -- (e);
		\draw[edge] (d1) -- (e);

		\draw[edge] (c3) -- (f);
		\draw[edge] (d2) -- (f);
		\draw[edge] (d3) -- (f);

		\draw[edge] (e) -- (top);
		\draw[edge] (f) -- (top);

	\end{tikzpicture}

	\caption{A finite graded regular poset; $\mu$-regularity fails with $X_3$ in place of $\Omega$.}
	\label{fig:counterex_truncation}
\end{figure}

\begin{lemma}\label{lem:level_sizes}
	Let $(X,{\leq})$ be a finite regular poset.
	Then, for $r\in\{0,\ldots,n\}$, we have
	\[
		\#X_0 = 1\text{,}\qquad
		\#X_r = \theta(0,r) = \frac{\theta(0)}{\theta(r)} \mu(0,r) = \frac{\theta(0)}{\theta(r)} \nu(r,n)
		\text{,}\qquad\text{and}\qquad
		\#\Omega = \theta(0)\text{.}
	\]
\end{lemma}

\begin{proof}
	Since $X_0=\{\bot\}$, we have $\#X_0=1$, and $\#\Omega=\theta(0)$ by definition.
	Moreover, $\#X_r=\theta(0,r)$ follows from Lemma~\ref{lem:theta_double} with $a=\bot$.
\end{proof}

We extend $\nu$ to the domain $\{0,\ldots,n\}^2$ by setting $\nu(r,s)=0$ whenever $r>s$.
The matrix $(\nu) = (\nu(r,s))_{r,s\in\{0,\ldots,n\}}$ is upper triangular with ones on the diagonal.
Hence, it is invertible, and its inverse $(\nu'(r,s))_{r,s\in\{0,\ldots,n\}} = (\nu)^{-1}$ is again upper triangular with ones on the diagonal.
Moreover, since the $\nu$-matrix has determinant $1$, all entries $\nu'(r,s)$ of its inverse are integers.
The matrices $(\nu)$ and $(\nu')$ lead to the following principle of $\nu$-inversion.

\begin{lemma}[$\nu$-inversion]\label{lem:nu_inversion}
	Let $r\in\{0,\ldots,n\}$ and $\vek{v}, \vek{w} \in \C^{r+1}$.
	Then
	\[
                \forall i\in\{0,\ldots,r\} : \sum_{j=i}^r\nu(i,j) v_j = w_i
                \iff
                \forall j\in\{0,\ldots,r\} : \sum_{i=j}^r\nu'(j,i) w_i = v_j\text{.}
        \]
\end{lemma}

\begin{proof}
	Let $A$ be the top left $(r+1)\times (r+1)$ submatrix of $(\nu)$.
	By the triangular form of $(\nu)$, the matrix $A$ is invertible, and its inverse $A^{-1}$ is the top left $(r+1)\times (r+1)$ submatrix of $(\nu')$.
	Now the statement reduces to $A \vek{v} = \vek{w} \iff \vek{v} = A^{-1} \vek{w}$.
\end{proof}

The regular poset setting of this section has its origin in the article~\cite{Delsarte-1976-JCTSA20[2]:230-243}, where the goal was to obtain an association scheme on the top level $\Omega$.
To this end, an additional regularity property is needed.
For completeness, we give a brief outline of that theory, even though our article does not require an association scheme structure.

\begin{definition}\label{def:delsarte_regular}
	Let $(X,\leq)$ be a finite $\wedge$-semilattice with bottom element $\bot$ and let $n = h(X)$.
	We define the following regularity property.
	\begin{enumerate}
		\customitem{(R$\pi$)}\label{regular:pi}
		For all $r,s,t\in\{0,\ldots,n\}$ for which there exists a pair $(a,b)\in (X_s \times \Omega)$ with $a \wedge b \in X_r$, the number
		\[
			\pi(r,s,t) = \#\{(x,y) \in X_t \times \Omega \mid a \leq y\text{ and }x\leq y\text{ and }x \leq b\}
		\]
		is constant over all choices of $(a,b)\in (X_s \times \Omega)$ with $a \wedge b \in X_r$.
	\end{enumerate}
	A finite poset $(X,\leq)$ is called \emph{Delsarte-regular} if it is a regular $\wedge$-semilattice satisfying, in addition, the regularity property~\ref{regular:pi}.
\end{definition}

\begin{fact}[{{\cite[Th.~7]{Delsarte-1976-JCTSA20[2]:230-243}}}]\label{fct:pi_implies_ass_scheme}
Let $(X,{\leq})$ be a Delsarte-regular finite $\wedge$-semilattice of height $n = h(X)$.
Then the set $\mathcal{R}$ of all non-empty relations%
\footnote{
	The reverse order induced by the term $n-i$ matches the metric structure of many association schemes obtained this way.
}
\[
	x \mathrel{R_i} y \iff h(x\wedge y) = n-i
	\qquad (i\in\{0,\ldots,n\})
\]
is a symmetric association scheme on $\Omega$.
\end{fact}

\begin{remark}\label{rem:short}
	A finite regular $\wedge$-semilattice $(X,{\leq})$ is called \emph{short} when each $a\in X$ has the form $a = x \wedge y$ with $x,y\in \Omega$.
	If $(X,{\leq})$ is a short Delsarte-regular semilattice, then all relations $R_i$ defined above are non-empty, and hence the association scheme $\mathcal{R} = \{R_0,\ldots,R_n\}$ produced by Fact~\ref{fct:pi_implies_ass_scheme} has $n$ classes.
\end{remark}

\section{The classical meet-semilattices in combinatorics}\label{sec:classical_lattices}

In \cite{Delsarte-1976-JCTSA20[2]:230-243}, Delsarte considered four prominent families of regular $\wedge$-semilattices, which may be regarded as the \enquote{big four} semilattices of algebraic combinatorics.
They come in two pairs, each consisting of a classical family based on the subset lattice of a base set $V$, together with its $q$-analog, obtained by replacing the subset lattice by the lattice of subspaces of an ambient vector space $V$ over~$\F_q$ and interpreting the same lattice-theoretic notions in this new setting.%
\footnote{
        Subsets of $V$ are replaced by subspaces, their unions by sums of subspaces (each the join in the corresponding lattice), their cardinalities by dimensions (the heights), and so on.
}
The first pair, of \emph{triangular type} (\emph{Type~I} in Delsarte's terminology), consists of the semilattices underlying the Johnson and $q$-Johnson schemes.
The second pair, of \emph{hypercubic type} (\emph{Type~II}), consists of the semilattices underlying the Hamming and $q$-Hamming schemes.
For many counting formulas, the classical family and its $q$-analog can be treated simultaneously:
the formula for the classical case is obtained as the formal specialization $q=1$ of the corresponding $q$-analog formula.%
\footnote{
	One reason is that, for fixed integers $a\geq 0$ and $b$, the Gaussian binomial coefficient $\qbinom{a}{b}{q}$ is a polynomial in $\Z[q]$.
	Its evaluation at $q=1$ recovers the ordinary binomial coefficient, i.\,e., $\qbinom{a}{b}{1} = \binom{a}{b}$.
}

As outlined in~\cite{Delsarte-1976-JCTSA20[2]:230-243}, all these semilattices are graded and Delsarte-regular, and their maximum height is the parameter $n$.
These families are more commonly referred to by the names of the association schemes on their top levels, arising from Fact~\ref{fct:pi_implies_ass_scheme}, such as the Johnson scheme or the Hamming scheme.
However, our theory does not require an association scheme on $\Omega$; indeed, the crucial structure is the underlying regular semilattice.
We therefore use the corresponding semilattice terminology, speaking of the \emph{Johnson semilattice}, \emph{Hamming semilattice}, and so on.
For each family, we also indicate when the semilattice is short, since then the resulting association scheme is non-degenerate in the sense of Remark~\ref{rem:short}.

Further families of regular $\wedge$-semilattices can be found in \cite[Sec.~2]{Suda-2012-DM312[10]:1827-1831} and \cite[Sec.~3]{Guo-Ma-Wang-2013-SChinaM56[11]:2393-2407}.

\subsection{Triangular type: Johnson and $q$-Johnson semilattices}

Let $V$ be a finite set of size $v$ and fix an integer $n\in\{0,\ldots,v\}$.
The \emph{Johnson semilattice} $J(v,n)$ consists of the set $X$ of all subsets of $V$ of size at most $n$, ordered by set inclusion.
In other words, it is the subset lattice of $V$, truncated at cardinality $n$.
For $A\in X$, we have $h(A) = \#A$.

In the corresponding $q$-analog setting, let $V$ be an $\F_q$-vector space of finite dimension $v$ and fix an integer $n\in \{0,\ldots,v\}$.
The \emph{$q$-Johnson semilattice} $J_q(v,n)$ consists of the set $X$ of all subspaces of $V$ of dimension at most $n$, ordered by inclusion.
In other words, it is the subspace lattice of $V$, truncated at dimension $n$.
For $A\in X$, we have $h(A) = \dim(A)$.

With the ordinary case corresponding to $q=1$, as discussed above, the regularity parameters of the semilattices $J(v,n)$ and $J_q(v,n)$ are
\begin{align*}
	\theta(r,s) & = \qbinom{v-r}{s-r}{q}\text{,} &
	\mu(r,s) & = \qbinom{n-r}{s-r}{q}\text{,} &
	\nu(r,s) & = \qbinom{s}{r}{q}\text{.} &
\end{align*}

The Johnson and $q$-Johnson semilattices are short if and only if $n \leq \frac{v}{2}$.

\subsection{Hypercubic type: Hamming and $q$-Hamming semilattices}\label{subsec:classical_lattices:hypercubic}
Let $N$ and $M$ be finite sets of sizes $n$ and $m \geq 1$, respectively.
The \emph{Hamming semilattice} $H(m,n)$ is the poset of all partial maps $N \to M$.
More explicitly, $X$ consists of all maps $f : W_f \to M$ with $W_f \subseteq N$, and for $f,g\in X$ we have $f \leq g$ if and only if $W_f \subseteq W_g$ and $g|_{W_f} = f$.
The height function of this semilattice is $h(f) = \#W_f$.
Moreover, the meet $f \wedge g$ has domain
\[
	W_{f\wedge g} = \{x\in W_f \cap W_g \mid f(x) = g(x)\}
\]
and is given by
\[
	(f\wedge g)(x) = f(x) = g(x)\text{.}
\]
Equivalently, one may adjoin a new symbol $\ast$ to $M$, with the meaning \enquote{not defined}, and identify the elements $f\in X$ with maps $N \to M \cup \{{\ast}\}$ by setting $f(x) = {\ast}$ for all $x\in N \setminus W_f$.%
\footnote{
	We prefer the description in terms of variable domains $W_f$, since it is more natural for forming the $q$-analog.
}

The top level $\Omega$ consists of all maps $f : N \to M$.
It can be identified with the usual Hamming space in coding theory as follows:
the set $M$ is regarded as an alphabet and $N$ as the set of positions, typically $N = \{1,\ldots,n\}$, so that maps $N \to M$ are the same as words in $M^n$.
Moreover, the words corresponding to two maps $f, g : N \to M$ are at Hamming distance $d$ if and only if $h(f \wedge g) = n-d$.

In the corresponding $q$-analog setting, let $N$ and $M$ be finite $\F_q$-vector spaces of dimensions $n$ and $m$, respectively.
The \emph{$q$-Hamming semilattice} $H_q(m,n)$ is the poset of all partial linear maps $N \to M$.
More explicitly, $X$ consists of all linear maps $f : W_f \to M$ with a subspace $W_f \subseteq N$, with the order and meet as above.
The height function of this semilattice is $h(f) = \dim(W_f)$.
The top level $\Omega$ consists of all linear maps $f : N \to M$.
This corresponds to the metric space underlying $q$-ary $m\times n$ rank-metric codes:
after fixing bases of $N$ and $M$, $\Omega$ can be identified with the set of all $m \times n$ matrices over $\F_q$.
Furthermore, the matrices corresponding to two linear maps $f, g : N \to M$ are at rank distance $d$, i.\,e., $\rk(f - g) = d$, if and only if $h(f \wedge g) = n-d$.%
\footnote{This follows from $W_{f\wedge g} = \ker(f-g)$ and the rank-nullity theorem.}

The regularity parameters of the semilattices $H(m,n)$ and $H_q(m,n)$ are
\begin{align*}
	\theta(r,s) & = (\#M)^{s-r}\qbinom{n-r}{s-r}{q}\text{,} &
	\mu(r,s) & = \qbinom{n-r}{s-r}{q}\text{,} &
	\nu(r,s) & = \qbinom{s}{r}{q}\text{.} &
\end{align*}

The coincidence of the formulas for $\mu$ and $\nu$ in the Johnson and Hamming cases, and likewise their $q$-analogs, reflects the fact that both parameters count elements below a fixed map $f$, possibly subject to further restrictions.
Hence, the values of the maps being counted are already forced, and only the underlying domains $W_g$ have to be chosen.
Consequently, when viewed downward from a fixed element, the Hamming and $q$-Hamming semilattices look just like their Johnson counterparts.
The difference appears only when extending upward, where the freedom of choosing new values produces the extra factor $(\#M)^{s-r}$ in $\theta(r,s)$.

For the ordinary Hamming semilattice with $n > 0$, the requirement $m \geq 1$ is necessary for the intended maximum height to be $n$.%
\footnote{
	If $m=0$ and $n > 0$, there are no total maps $N \to M$.
}
The Hamming semilattice is short precisely for $m \geq 2$, together with the exceptional case $(n,m) = (0,1)$.
The $q$-Hamming semilattice is short if and only if $m \geq n$.

\begin{remark}
	The formula $\nu(r,s) = \qbinom{s}{r}{q}$ applies to all four discussed semilattices.
	In~\cite[p.~470]{Kiermaier-Pavcevic-2015-JCD23[11]:463-480}, the matrix $(\nu)$ is called the upper triangular $q$-Pascal matrix, and the entries of its inverse are computed as
	\[
		\nu'(r,s) = (-1)^{s-r} q^{\binom{s-r}{2}} \qbinom{s}{r}{q}\text{.}
	\]
\end{remark}

\section{Designs in regular posets}\label{sec:poset_design}
We will develop the theory of designs for general complex-valued functions $\beta:\Omega\to\mathbb{C}$, which can be understood as attaching weights to the elements of $\Omega$.
The classical \emph{set case}, where one works with a subset $A\subseteq \Omega$ rather than a weight function, is included as the special case $\beta=\chi_A$, where $\chi_A : \Omega \to\{0,1\}$ denotes the characteristic function of $A$.
To simplify notation, we will freely identify a subset $A \subseteq \Omega$ with its characteristic function $\chi_A$.

In what follows, $\C^{\Omega}$ is equipped with the standard Hermitian inner product $(\beta, \gamma) \mapsto \langle \beta, \gamma\rangle \coloneqq \beta^{\ast} \gamma$, where $\beta^{\ast}$ is the conjugate transpose of $\beta$.
As usual, we write $\norm{\beta} = \sqrt{\beta^{\ast} \beta}$, and we use the abbreviation
\[
	\#\beta = \sum_{x\in \Omega} \beta(x) = \vek{1}_{\Omega}^{\ast} \beta = \overline{\beta^{\ast} \vek{1}_{\Omega}}\text{,}
\]
with the idea that in the set case $A \subseteq \Omega$, the size of $A$ equals $\#\chi_A$.

\begin{definition}
	Let $\beta \in \C^\Omega$.
	For $a\in X$, we define
	\[
		\lambda^{(\beta)}(a) = \lambda(a) = \sum_{\substack{x\in \Omega \\ a \leq x}} \beta(x)\text{.}
	\]
	By the identification of subsets $A \subseteq \Omega$ with their characteristic functions $\chi_A$, we use the notation
	\[
		\lambda^{(A)}(a) = \lambda^{(\chi_A)}(a)\text{.}
	\]
\end{definition}

Clearly, $\lambda^{(\beta)}(\bot) = \#\beta$.

\begin{definition}\label{def:poset_design}
	Let $\beta \in \C^\Omega$ and $t\in\{0,\ldots,n\}$.
	If $\lambda \coloneqq \lambda^{(\beta)}(a)$ is constant over all choices of $a\in X_t$, the function $\beta$ is called a \emph{design} of \emph{strength} $t$ or, in short, a \emph{$t$-design}.
	The parameter $\lambda$ is called the \emph{index} of the $t$-design $\beta$.

	A subset $D\subseteq \Omega$ is identified with its characteristic function $\chi_D$.
	Accordingly, $D$ is a $t$-\emph{design} if $\chi_D$ is a $t$-design in the above sense.
\end{definition}

\begin{example}
	For every $t\in\{0,\ldots,n\}$, the top level $\Omega$ is a $t$-design of index $\lambda = \theta(t)$.
	It is called the \emph{complete} $t$-design; its index is often denoted as $\lambda_{\max}$ in the literature, since in the set case, it is the maximum possible index for a $t$-design.
\end{example}

\begin{fact}[{{\cite[Prop.~2.4]{Suda-2012-DM312[10]:1827-1831}}}]\label{fct:general:lambda_i}
	Let $\beta\in \C^{\Omega}$ be a $t$-design of index $\lambda$.
	Then $\beta$ is an $s$-design for all $s\in\{0,\ldots,t\}$.
	The associated index is
	\[
		\lambda_s = \lambda \frac{\theta(s)}{\theta(t)}\text{.}
	\]
	In particular,
	\[
		\#\beta = \lambda_0 = \lambda \frac{\theta(0)}{\theta(t)}\text{.}
	\]
\end{fact}

\begin{proof}
	Let $a\in X_s$.
	Then
	\begin{multline*}
		\lambda \theta(s,t)
		= \sum_{\substack{x\in X_t \\ a \leq x}} \lambda
		= \sum_{\substack{x\in X_t \\ a \leq x}} \sum_{\substack{y\in \Omega \\ x \leq y}} \beta(y)
		= \sum_{\substack{y\in \Omega \\ a \leq y}} \beta(y) \sum_{\substack{x\in X_t\\a \leq x \leq y}} 1 \\
		= \sum_{\substack{y\in \Omega \\ a \leq y}} \beta(y) \mu(s,t)
		= \mu(s,t) \sum_{\substack{y\in \Omega \\ a \leq y}} \beta(y)
		= \mu(s,t) \lambda(a)\text{.}
	\end{multline*}
	Hence
	\[
		\lambda(a)
		= \lambda \frac{\theta(s,t)}{\mu(s,t)}
		= \lambda \frac{\theta(s)}{\theta(t)}\text{,}
	\]
	where we used Lemma~\ref{lem:theta_double} in the last step.
\end{proof}

\begin{remark}
The statement of Fact~\ref{fct:general:lambda_i} is first found in~\cite[p.~240, Remark]{Delsarte-1976-JCTSA20[2]:230-243}, where it was shown in an algebraic way, based on the dual distribution of $\beta$, and hence needs the stronger assumption of Delsarte-regularity.
The above combinatorial proof is essentially the one from \cite[Prop.~2.4]{Suda-2012-DM312[10]:1827-1831}.
It generalizes the double-counting argument from ordinary design theory, and can also be understood as a streamlined version of the proof of~\cite[Th.~4.33]{Bannai-Bannai-Ito-Tanaka-2021-AlgebraicCombinatorics}.
\end{remark}

\begin{theorem}\label{thm:divisibility}
	Let $\beta\in\Z^{\Omega}$ be an integer-valued $t$-design of index $\lambda$.
	Then the following equivalent statements hold:
	\begin{enumerate}[(a)]
		\item\label{thm:divisibility:integrality}
		\textit{Integrality conditions.}\nobreak\hspace{1em}\ignorespaces
		For all $s\in\{0,\ldots,t\}$,
		\[
			\lambda_s \in \Z\text{.}
		\]
		\item\label{thm:divisibility:vanilla}
		\textit{Divisibility conditions.}\nobreak\hspace{1em}\ignorespaces
		For all $s\in\{0,\ldots,t\}$,
		\[
			\theta(t) \mid \lambda \theta(s)\text{.}
		\]
		\item\label{thm:divisibility:combined}
		\[
			\frac{\theta(t)}{\gcd(\theta(0),\ldots,\theta(t))} \mid \lambda\text{.}
		\]
	\end{enumerate}
\end{theorem}

\begin{proof}
	Let $s\in\{0,\ldots,t\}$.
	By Fact~\ref{fct:general:lambda_i}, $\beta$ is an $s$-design of index $\lambda_s$.
	Fixing some $a\in X_s$,
	\[
		\lambda_s
		= \sum_{\substack{x\in \Omega\\a \leq x}}\beta(x)\in\Z\text{,}
	\]
	so statement~\ref{thm:divisibility:integrality} holds.

	The equivalence of~\ref{thm:divisibility:integrality} and~\ref{thm:divisibility:vanilla} follows from the expression (see Fact~\ref{fct:general:lambda_i}) $\lambda_s = \lambda \frac{\theta(s)}{\theta(t)}$.
	The equivalence of~\ref{thm:divisibility:vanilla} and~\ref{thm:divisibility:combined} is elementary number theory.
\end{proof}

\begin{example}[Designs in the four classical families of semilattices]\label{ex:designs_classical_semilattices}\thmheadernewline
	In the Johnson semilattice $J(v,n)$, designs of strength $t$ and index $\lambda$ are precisely combinatorial $t$-$(v,n,\lambda)$ block designs without repeated blocks; in the $q$-Johnson semilattice, they are precisely $t$-$(v,n,\lambda)_q$ subspace designs without repeated blocks.
	Standard references for the theory of block designs include \cite[Ch.~19]{VanLint-Wilson-2001-Combinatorics} and the two-volume treatise \cite{Beth-Jungnickel-Lenz-1999-DesignTheoryI,Beth-Jungnickel-Lenz-1999-DesignTheoryII}; for subspace designs, see \cite{Braun-Kiermaier-Wassermann-2018-SignalsCommunTechnol:171-211}.
	For combinatorial designs, Fact~\ref{fct:general:lambda_i} and Parts~\ref{thm:divisibility:integrality} and~\ref{thm:divisibility:vanilla} of Theorem~\ref{thm:divisibility} are classical; for subspace designs, see~\cite[Lem.~4.1(1)]{Suzuki-1990-EuJC11[6]:601-607} and~\cite[Lem.~2.3]{Kiermaier-Laue-2015-AiMoC9[1]:105-115}.
	For both combinatorial and subspace designs, Theorem~\ref{thm:divisibility}\ref{thm:divisibility:combined} appears in~\cite[Lem.~2.6]{Kiermaier-Mannaert-Wassermann-2025-JCTSA212:P105979}.

	In the Hamming semilattice $H(m,n)$, $t$-designs of index $\lambda$ correspond to orthogonal arrays with parameters $\OA(\lambda m^t,n,m,t)$ \cite[Th.~4.4]{Delsarte-1973-PhilRRSuppl10}; for a standard reference on orthogonal arrays, see the book~\cite{Hedayat-Sloane-Stufken-1999-OrthogonalArrays}.
	General designs in the $q$-Hamming semilattices $H_q(m,n)$ seem to have received comparatively little attention in the literature.
	Perhaps the reason is that, unlike in the other three families of regular semilattices, Steiner systems, i.e., designs of index $1$, always exist when $H_q(m,n)$ is short, making the study of designs of index $\lambda \geq 2$ significantly less attractive.
	For the Hamming and $q$-Hamming semilattices, Theorem~\ref{thm:divisibility}\ref{thm:divisibility:combined} yields $1 \mid \lambda$, so the integrality conditions are trivial.
\end{example}

\begin{remark}\label{rem:general_without_nu}
	All statements in this section that do not explicitly involve the function~$\nu$ remain valid without assuming the regularity property~\ref{regular:nu}.
	In particular, Lemma~\ref{lem:theta_double} and Fact~\ref{fct:general:lambda_i} rely only on~\ref{regular:theta} and~\ref{regular:mu}.
\end{remark}

\section{Mendelsohn equations in regular semilattices}\label{sec:semilattice_mendelsohn}
Since the original Mendelsohn equations arise from the study of intersections, we now equip the poset $(X,\le)$ with the structure of a $\wedge$-semilattice.
We assume that $X$ is finite and non-empty; in this case, the semilattice structure forces a bottom element, namely
\[
	\bot = \bigwedge_{x\in X} x\text{.}
\]
We further assume that $(X,{\leq})$ satisfies the regularity properties~\ref{regular:theta}, \ref{regular:mu}, and~\ref{regular:nu}, so that the present setting is a refinement of that considered in Section~\ref{sec:poset_design}.

\begin{definition}
	For $r\in\{0,\ldots,n\}$, the \emph{intersection numbers} of $\beta\in \C^{\Omega}$ with respect to $a\in X$ are defined as
	\[
		\alpha_r(a) = \sum_{\substack{y\in \Omega \\ h(y \wedge a) = r}} \beta(y)\text{.}
	\]
\end{definition}

\begin{lemma}\label{lem:general:pre_mendelsohn}
	Let $\beta\in \C^{\Omega}$ and $a\in X$.
	Then for all $r\in\{0,\ldots,n\}$,
	\[
		\sum_{s=r}^n \nu(r,s)\alpha_s(a) = \sum_{\substack{x\in X_r \\ x \leq a}} \lambda(x)\text{.}
	\]
\end{lemma}

\begin{proof}
	For all $r\in\{0,\ldots,n\}$,
	\[
		\sum_{s=r}^n \alpha_s(a)\nu(r,s)
		= \sum_{s=r}^n \sum_{\substack{y\in \Omega \\ h(y \wedge a) = s}} \beta(y)\sum_{\substack{x\in X_r\\ x \leq y\wedge a}} 1
		= \sum_{\substack{x\in X_r \\ x \leq a}} \sum_{\substack{y\in \Omega \\ x \leq y}}\beta(y)
		= \sum_{\substack{x\in X_r \\ x \leq a}} \lambda(x)\text{.}
	\]
\end{proof}

\begin{theorem}\label{thm:general:mendelsohn}
	Let $\beta\in \C^{\Omega}$ be a $t$-design, and let $a\in X$.
	Then the following systems of equations hold, and they are equivalent.
	\begin{enumerate}[(a)]
		\item\label{thm:general:mendelsohn:mendelsohn}\emph{Mendelsohn equations.}\nobreak\hspace{1em}\ignorespaces
		For all $r\in\{0,\ldots,t\}$,
		\[
			\sum_{s=r}^n \nu(r,s) \alpha_s(a) = \nu(r,h(a)) \lambda_r\text{.}
		\]
		The right-hand side admits the alternative expressions
		\[
			\nu(r,h(a)) \lambda_r
			= \nu(r,h(a)) \frac{\theta(r)}{\theta(t)} \lambda
			= \nu(r,h(a)) \frac{\theta(r)}{\theta(0)} \#\beta\text{.}
		\]
		\item\label{thm:general:mendelsohn:koehler}\emph{Köhler parametrization.}\nobreak\hspace{1em}\ignorespaces
		For all $r\in\{0,\ldots,t\}$,
		\[
			\alpha_r(a)
			= \sum_{s=r}^t \nu'(r,s)\nu(s,h(a)) \lambda_s - \sum_{u=t+1}^n \Big(\sum_{s=r}^t \nu'(r,s)\nu(s,u)\Big)\alpha_u(a)\text{,}
		\]
		where the inner sum on the right-hand side admits the alternative form
		\[
			\sum_{s=r}^t \nu'(r,s)\nu(s,u) = -\sum_{s=t+1}^u \nu'(r,s)\nu(s,u)\text{.}
		\]
	\end{enumerate}
\end{theorem}

\begin{proof}
	Part~\ref{thm:general:mendelsohn:mendelsohn}:
	For all $r\in\{0,\ldots,t\}$,
	\[
		\sum_{s=r}^n \alpha_s(a)\nu(r,s)
		\overset{\text{Lem.~\ref{lem:general:pre_mendelsohn}}}{=} \sum_{\substack{x\in X_r \\ x \leq a}} \lambda(x)
		\overset{\text{Fact~\ref{fct:general:lambda_i}}}{=} \sum_{\substack{x\in X_r \\ x \leq a}} \lambda_r
		= \nu(r,h(a))\lambda_r\text{.}
	\]
	The alternative expressions follow from Fact~\ref{fct:general:lambda_i}.

	Part~\ref{thm:general:mendelsohn:koehler}:
	Rearranging the equations in Part~\ref{thm:general:mendelsohn:mendelsohn} yields
	\[
		\sum_{s=r}^t \nu(r,s) \alpha_s(a) = \nu(r,h(a))\lambda_r - \sum_{u=t+1}^n \nu(r,u) \alpha_u(a)\qquad\text{for all }r\in\{0,\ldots,t\}\text{.}
	\]
	By the principle of $\nu$-inversion from Lemma~\ref{lem:nu_inversion}, this is equivalent to
	\[
		\alpha_r(a) = \sum_{s=r}^t \nu'(r,s)\Big(\nu(s,h(a))\lambda_s - \sum_{u=t+1}^n \nu(s,u) \alpha_u(a)\Big)\qquad\text{for all }r\in\{0,\ldots,t\}\text{.}
	\]
	Expanding the right-hand side, we obtain the claimed expression.
	\[
		\alpha_r(a) = \sum_{s=r}^t \nu'(r,s)\nu(s,h(a))\lambda_s - \sum_{u=t+1}^n \Big(\sum_{s=r}^t \nu'(r,s)\nu(s,u)\Big) \alpha_u(a)\text{.}
	\]

	For the alternative form of the inner sum, we use that $(\nu)$ and $(\nu')$ are upper triangular and mutually inverse, and that $r \le t < u$.
	Thus,
	\[
		\sum_{s=r}^n \nu'(r,s)\nu(s,u)
		= \sum_{s=0}^n \nu'(r,s) \nu(s,u)
		= \delta_{ru}
		= 0\text{,}
	\]
	and therefore
	\[
		\sum_{s=r}^t \nu'(r,s)\nu(s,u)
		= - \sum_{s=t+1}^n \nu'(r,s) \nu(s,u)
		= - \sum_{s=t+1}^u \nu'(r,s) \nu(s,u)\text{.}
	\]
\end{proof}

\begin{remark}
	Part~\ref{thm:general:mendelsohn:koehler} of Theorem~\ref{thm:general:mendelsohn} is a linear transform of the system in Part~\ref{thm:general:mendelsohn:mendelsohn}, expressing the numbers $\alpha_0(a),\ldots,\alpha_t(a)$ in terms of $\alpha_{t+1}(a),\ldots,\alpha_n(a)$.
	This can be advantageous, especially in situations where only a few of the numbers $\alpha_{t+1}(a),\ldots,\alpha_n(a)$ are non-zero.
	In particular, this occurs when $t$ is close to $n$, and also for Steiner systems $D$, which is the basis for Theorem~\ref{thm:general:steiner_block_intersection_distribution} below.
\end{remark}

\begin{remark}
	We investigate the specialization of Theorem~\ref{thm:general:mendelsohn} to our standard regular semilattices.
	For Johnson schemes, i.\,e., for combinatorial designs $D \subseteq \binom{V}{k}$, Part~\ref{thm:general:mendelsohn:mendelsohn} of Theorem~\ref{thm:general:mendelsohn} was shown for blocks $a\in D$ in~\cite{Mendelsohn-1969-NAMS16[6]:984,Mendelsohn-1971} and, for arbitrary subsets $a\subseteq V$, in~\cite[Eq.~3]{Goethals-1970-InstStatistMimeoSer600_29} and independently in \cite[Satz~2]{Oberschelp-1972-MPSemBNF19:55-67}.
	Part~\ref{thm:general:mendelsohn:koehler} appears in~\cite[Satz~1]{Koehler-1989-DM73[1-2]:133-142} and, independently and with a considerably cleaner proof, in~\cite[Satz~1]{Harnau-1988-RostMKoll34:47-52}.
	For $q$-Johnson schemes, i.\,e., $q$-analogs of designs, both parts are contained in~\cite[Th.~2.4 and Th.~2.6]{Kiermaier-Pavcevic-2015-JCD23[11]:463-480}.
	For Hamming schemes, i.\,e., orthogonal arrays, Theorem~\ref{thm:general:mendelsohn}\ref{thm:general:mendelsohn:mendelsohn} reproduces the system of equations in~\cite[Eq.~(3.0b)]{Bose-Bush-1952-AnnMathStatistics23[4]:508-524} (see also~\cite[Lemma~2.7]{Hedayat-Sloane-Stufken-1999-OrthogonalArrays}), which has been used to establish the non-existence of orthogonal arrays with certain parameters, see for example Theorem~2A and Theorem~2B in~\cite{Bose-Bush-1952-AnnMathStatistics23[4]:508-524}.
	We are not aware of a published version of Part~\ref{thm:general:mendelsohn:koehler} for Hamming schemes, nor of either part for $q$-Hamming schemes.
	Hence, we state the corresponding specialization below as Corollary~\ref{cor:hamming:mendelsohn}.
\end{remark}

\begin{corollary}[{{Mendelsohn equations in hypercubic semilattices}}]\label{cor:hamming:mendelsohn}\thmheadernewline
	Let $(X,{\leq}) = H_q(m,n)$ be a Hamming or a $q$-Hamming semilattice, where possibly $q=1$, and let $\Omega$ be its top level.
	Let $\beta : \Omega \to \C$ be a $t$-design of index $\lambda$.
	Let $a\in X$.
	With the height function $h$ and the symbol $M$ as in Section~\ref{subsec:classical_lattices:hypercubic}, the following systems of equations hold, and they are equivalent.
	\begin{enumerate}[(a)]
		\item\label{cor:hamming:mendelsohn:mendelsohn}
			\emph{Mendelsohn equations.}\nobreak\hspace{1em}\ignorespaces
			For all $r\in\{0,\ldots,t\}$,
			\[
				\sum_{s=r}^n \qbinom{s}{r}{q}\alpha_s(a)
				= \qbinom{h(a)}{r}{q} (\#M)^{t-r}\lambda\text{.}
			\]
			The right-hand side admits the alternative expression
			\[
				\qbinom{h(a)}{r}{q} (\#M)^{t-r}\lambda
				=
				\qbinom{h(a)}{r}{q}(\#M)^{-r}\#\beta\text{.}
			\]
		\item\label{cor:hamming:mendelsohn:koehler}
			\emph{Köhler parametrization.}\nobreak\hspace{1em}\ignorespaces
			For all $r\in\{0,\ldots,t\}$,
			\begin{align*}
				\alpha_r(a)
				& = \qbinom{h(a)}{r}{q}(\#M)^{t-r}\lambda \sum_{j=0}^{t-r} (-1)^j q^{\binom{j}{2}} \qbinom{h(a)-r}{j}{q}(\#M)^{-j} \\
				& {\phantom{~=~}} -\sum_{u=t+1}^n \qbinom{u}{r}{q} \left( \sum_{j=0}^{t-r} (-1)^j q^{\binom{j}{2}} \qbinom{u-r}{j}{q} \right) \alpha_u(a)\text{.}
			\end{align*}
			The inner sum on the right-hand side admits the alternative form
			\[
				\sum_{j=0}^{t-r} (-1)^j q^{\binom{j}{2}} \qbinom{u-r}{j}{q}
				= - \sum_{j=t-r+1}^{u-r} (-1)^j q^{\binom{j}{2}} \qbinom{u-r}{j}{q}\text{.}
			\]
	 \end{enumerate}
\end{corollary}

\begin{proof}
	We specialize Theorem~\ref{thm:general:mendelsohn} to the Hamming and $q$-Hamming semilattices.
	Part~\ref{cor:hamming:mendelsohn:mendelsohn} follows directly.
	For Part~\ref{cor:hamming:mendelsohn:koehler}, we start with the formula in Theorem~\ref{thm:general:mendelsohn}\ref{thm:general:mendelsohn:koehler} and specialize the symbols $\theta$, $\mu$, $\nu$, and $\nu'$ to their expressions in the $q$-Hamming schemes.
	We then use the trinomial identity
	\[
		\qbinom{x}{y}{q} \qbinom{y}{z}{q} = \qbinom{x}{z}{q}\qbinom{x-z}{y-z}{q}
	\]
	in the $s$-sums, move the factor independent of $s$ out of the summations, and finally reindex by setting $j = s-r$.
\end{proof}

\section{Codes and Steiner systems}\label{sec:steiner}

This section treats the special case of Steiner systems, which occupy a distinguished place in the theory:
they exhibit strong structural properties, such as the constant block intersection numbers in Theorem~\ref{thm:general:steiner_block_intersection_distribution}.
At the same time, through their correspondence with certain extremal codes (Theorem~\ref{thm:singleton}), they are also of significant relevance in coding theory.

For these considerations, it is natural to focus on the set case.
Thus, we investigate subsets $A \subseteq \Omega$, identified with their characteristic functions $\chi_A : \Omega \to \{0,1\}$, rather than arbitrary functions $\beta : \Omega \to \C$.

\begin{definition}\label{def:steiner_code}
	Let $(X,{\leq})$ be a finite poset with bottom element.
	Let $A \subseteq \Omega$.
	\begin{enumerate}[(a)]
	\item\label{def:steiner_code:steiner}
		The set $A$ is called a \emph{Steiner system} of strength $t \in \{0,\ldots,n\}$ if $\lambda^{(A)}(x) = 1$ for all $x\in X_t$.
		Equivalently, $A$ is a $t$-design of index $\lambda = 1$.
	\item\label{def:steiner_code:code}
		The set $A$ is called a \emph{$d$-code} for $d\in\{1,\ldots,n\}$ if $\lambda^{(A)}(x) \leq 1$ for all $x\in X_{n-d+1}$.
		Furthermore, every subset of $\Omega$ is called a $0$-code, without further condition.
	\end{enumerate}
\end{definition}

\begin{remark}\label{rem:def_code}
	In the above definition, Steiner systems and codes were defined in parallel.
	As a result, codes are defined using the incidence condition of packing designs,%
	\footnote{
		$A \subseteq \Omega$ is called a \emph{packing design} of strength $t$ and index at most $\lambda$ if $\lambda^{(A)}(x)\leq \lambda$ for all $x\in X_t$.
		The complementary notion is that of a \emph{covering design}, where the condition $\lambda^{(A)}(x)\leq \lambda$ is replaced by $\lambda^{(A)}(x)\geq \lambda$.
	}
	which deviates from the usual coding-theoretic definition in terms of pairwise distances.
	In Lemma~\ref{lem:code_to_distance}, we show that this is nevertheless the expected coding-theoretic notion.
\end{remark}

\begin{lemma}\label{lem:lambda_monotone}
	Let $(X,{\leq})$ be a finite poset with bottom element.
	Let $A \subseteq \Omega$.
	Then for all $a,b\in X$ with $a \leq b$,
	\[
		\lambda^{(A)}(a) \geq \lambda^{(A)}(b)\text{.}
	\]
\end{lemma}

\begin{proof}
	For all $a,b\in X$ with $a \leq b$, we have
	\[
		\{x \in A \mid a \leq x\} \supseteq \{x\in A \mid b \leq x\}\text{,}
	\]
	and the claim follows by taking cardinalities.
\end{proof}

\begin{lemma}\label{lem:code_dist_monotonicity}
	Let $(X,{\leq})$ be a finite poset with bottom element.
	Let $A \subseteq \Omega$ be a $d$-code.
	Then $A$ is an $e$-code for all $e\in\{0,\ldots,d\}$.
\end{lemma}

\begin{proof}
	For $e = 0$, nothing is to be shown.
	For $e \geq 1$, let $y\in X_{n-e+1}$.
	By Lemma~\ref{lem:chain_Xr} and $e \leq d$, there exists an $x \in X_{n-d+1}$ with $x \leq y$.
	Now, by Lemma~\ref{lem:lambda_monotone}, we have
	\[
		\lambda^{(A)}(y) \leq \lambda^{(A)}(x) \leq 1\text{.}
	\]
\end{proof}

\begin{definition}\label{def:delta}
	Let $(X,{\leq})$ be a finite poset with bottom element.
	We define the map%
	\footnote{In this article, we use the convention $0\in\N$.}
	\[
		\delta : \Omega \times \Omega \to \N
	\]
	by setting $\delta(a,b)$ to be the smallest $r\in\N$ such that there exists an $x\in X_{n-r}$ with $x \leq a$ and $x\leq b$.
\end{definition}

\begin{remark}\label{rem:delta}
	The map $\delta$ is well-defined since $x = \bot$ is always a possible choice.
	If $(X,{\leq})$ is a $\wedge$-semilattice, then $\delta$ is the meet-height distance $\delta(a,b) = n - h(a \wedge b)$.

	Note that, in general, $\delta$ is not a metric on $\Omega$, even for short graded regular $\wedge$-semilattices.
	A suitable counterexample will be given in Example~\ref{ex:delta_not_metric}.
\end{remark}

\begin{lemma}\label{lem:code_to_distance}
	Let $(X,{\leq})$ be a finite poset with bottom element and let $A \subseteq \Omega$.
	\begin{enumerate}[(a)]
		\item\label{lem:code_to_distance:dist} Let $d\in\{0,\ldots,n\}$.
		Then $A$ is a $d$-code if and only if $\delta(a,b) \geq d$ for all distinct $a,b\in A$.
		\item\label{lem:code_to_distance:mindist} Assume that $\#A \geq 2$.%
		\footnote{This ensures that the minimum is taken over a non-empty set.}
		Then,
		\[
			\min_{\substack{a,b\in A \\ a\neq b}}\delta(a,b)
			= \max\{d\in\{0,\ldots,n\} \mid A \text{ is a }d\text{-code}\}\text{.}
		\]
		This number is called the \emph{minimum distance} of $A$.
	\end{enumerate}
\end{lemma}

\begin{proof}
	Part~\ref{lem:code_to_distance:dist}:
	For $d=0$, both conditions are automatic.
	Let $d\geq 1$.
	For \enquote{$\Rightarrow$}, let $A$ be a $d$-code and let $a,b\in A$ be distinct.
	By Lemma~\ref{lem:code_dist_monotonicity}, $A$ is an $e$-code for all $e\in\{1,\ldots,d\}$.
	For all these $e$, there does not exist an $x\in X_{n-e+1}$ with $x \leq a$ and $x \leq b$ by Definition~\ref{def:steiner_code}\ref{def:steiner_code:code}.
	Hence, by Definition~\ref{def:delta}, $\delta(a,b) \geq d$.
	For \enquote{$\Leftarrow$}, we argue by contraposition.
	Assume that $A$ is not a $d$-code.
	By Definition~\ref{def:steiner_code}\ref{def:steiner_code:code}, there exist distinct $a,b\in A$ and an $x\in X_{n-d+1}$ with $x \leq a$ and $x\leq b$.
	Now by Definition~\ref{def:delta}, $\delta(a,b) \leq d-1 < d$.

	Part~\ref{lem:code_to_distance:mindist} is now a direct consequence of Part~\ref{lem:code_to_distance:dist}.
\end{proof}

\begin{example}\label{ex:delta_not_metric}
	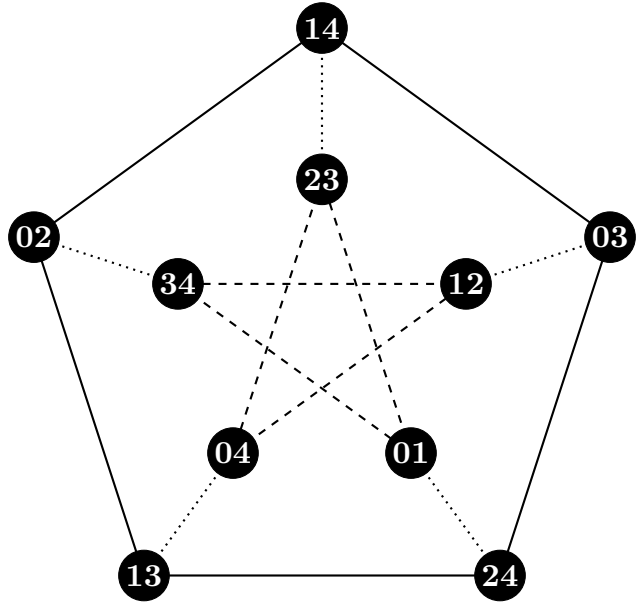
\begin{figure}
		\centering
		\begin{tikzpicture}
			\foreach \x in {0,...,4}{
			\coordinate (o\x) at (90+\x*72:4);
			\coordinate (i\x) at (90+\x*72:2);
			};

			\foreach \x in {0,...,4}{
			\draw[thick, dotted] (o\x) -- (i\x);
			};

			\draw[thick] (o0) -- (o1);
			\draw[thick] (o1) -- (o2);
			\draw[thick] (o2) -- (o3);
			\draw[thick] (o3) -- (o4);
			\draw[thick] (o4) -- (o0);

			\draw[thick, dashed] (i0) -- (i2);
			\draw[thick, dashed] (i1) -- (i3);
			\draw[thick, dashed] (i2) -- (i4);
			\draw[thick, dashed] (i3) -- (i0);
			\draw[thick, dashed] (i4) -- (i1);

			\node[circle,draw,minimum size=19pt,inner sep=0pt,fill] at (o0){$\color{white}\symbf{14}$};
			\node[circle,draw,minimum size=19pt,inner sep=0pt,fill] at (o1){$\color{white}\symbf{02}$};
			\node[circle,draw,minimum size=19pt,inner sep=0pt,fill] at (o2){$\color{white}\symbf{13}$};
			\node[circle,draw,minimum size=19pt,inner sep=0pt,fill] at (o3){$\color{white}\symbf{24}$};
			\node[circle,draw,minimum size=19pt,inner sep=0pt,fill] at (o4){$\color{white}\symbf{03}$};
			\node[circle,draw,minimum size=19pt,inner sep=0pt,fill] at (i0){$\color{white}\symbf{23}$};
			\node[circle,draw,minimum size=19pt,inner sep=0pt,fill] at (i1){$\color{white}\symbf{34}$};
			\node[circle,draw,minimum size=19pt,inner sep=0pt,fill] at (i2){$\color{white}\symbf{04}$};
			\node[circle,draw,minimum size=19pt,inner sep=0pt,fill] at (i3){$\color{white}\symbf{01}$};
			\node[circle,draw,minimum size=19pt,inner sep=0pt,fill] at (i4){$\color{white}\symbf{12}$};
		\end{tikzpicture}
		\caption{Petersen graph}\label{fig:petersen}
	\end{figure}

	As announced in Remark~\ref{rem:delta}, we construct a short graded regular $\wedge$-semilattice such that $\delta$ is not a metric on $\Omega$.
	Our example is derived from the Petersen graph.
	Consider its standard visualization depicted in Figure~\ref{fig:petersen}.
	The ten vertices are partitioned into five outer vertices and five inner vertices.
	The $15$ edges are partitioned into five outer edges forming a pentagon, i.e., the $5$-cycle formed by the solid lines; five inner edges forming a pentagram, shown as dashed lines; and five spokes, shown as dotted lines, each joining an outer vertex to an inner vertex.
	There are five pairs of opposite edges, each consisting of one outer edge and one inner edge, with the property that no endpoint of the outer edge is joined by a spoke to an endpoint of the inner edge.
	For instance, the edges $\{12,34\}$ and $\{13,24\}$ are opposite in Figure~\ref{fig:petersen}.
	We define
	\begin{align*}
		X_0 & = \{\bot\}\text{,} & & \#X_0 = 1\text{,} \\
		X_1 & = \text{all pairs of opposite edges,} & & \#X_1 = 5\text{,} \\
		X_2 & = \text{all outer and inner edges,} & & \#X_2 = 10\text{,} \\
		X_3 & = \text{all vertices,} & & \#X_3 = 10\text{,}
	\end{align*}
	and put $X = X_0 \cup X_1 \cup X_2 \cup X_3$.
	The order $\leq$ on $X$ is induced by incidence:
	$\bot$ is covered by every element of $X_1$, each opposite pair in $X_1$ is covered by its two edges in $X_2$, and each edge in $X_2$ is covered by its two endpoints in $X_3$.

	This definition can be made formally precise using the standard Kneser description of the Petersen graph based on the $5$-element set $\Z/5\Z$.
	As indicated in Figure~\ref{fig:petersen},%
	\footnote{In Figure~\ref{fig:petersen}, for the vertices, the shorthand $ij$ for $\{i,j\}$ is used.}
	its vertex set is $\binom{\Z/5\Z}{2}$, and two vertices are adjacent if and only if the corresponding $2$-subsets of $\Z/5\Z$ are disjoint.
	The outer and inner vertices can be represented as
	\[
		v^{(o)}_j = \{j-1,j+1\}
		\qquad\text{and}\qquad
		v^{(i)}_j = \{j-2,j+2\}\text{,}
	\]
	respectively, where $j\in\Z/5\Z$.%
	\footnote{Since $j\in\Z/5\Z$, expressions such as $j-2$ and $j+2$ are read modulo $5$.}
	Hence, the outer and inner edges have the form
	\[
		e^{(o)}_j = \{v^{(o)}_{j-2}, v^{(o)}_{j+2}\}
		\qquad\text{and}\qquad
		e^{(i)}_j = \{v^{(i)}_{j-1}, v^{(i)}_{j+1}\}\text{,}
	\]
	respectively.
	The spokes have the form $\{v^{(o)}_j, v^{(i)}_j\}$, and the opposite pairs of outer and inner edges have the form $\{e^{(o)}_j, e^{(i)}_j\}$.

	It is routine to check that $(X,{\leq})$ is a short graded regular $\wedge$-semilattice, with the following matrices of regularity parameters:
	\[
		(\theta) =
		\begin{pmatrix}
			1 & 5 & 10 & 10 \\
			0 & 1 & 2  & 4  \\
			0 & 0 & 1  & 2  \\
			0 & 0 & 0  & 1
		\end{pmatrix}\text{,}\qquad
		(\mu) = 
		\begin{pmatrix}
			1 & 2 & 2 & 1 \\
			0 & 1 & 1 & 1 \\
			0 & 0 & 1 & 1 \\
			0 & 0 & 0 & 1
		\end{pmatrix}\text{,}\qquad
		(\nu) = \begin{pmatrix}
			1 & 1 & 1 & 1 \\
			0 & 1 & 1 & 2 \\
			0 & 0 & 1 & 2 \\
			0 & 0 & 0 & 1
		\end{pmatrix}\text{.}
	\]
	We describe $a\wedge b$ for $a,b\in\Omega$, i.e., for two vertices $a,b$ of the Petersen graph.
	Clearly, $a\wedge a=a\in X_3$, and hence, $\delta(a,a) = 0$.
	For $a\neq b$, we first consider the case where $a$ and $b$ are both outer vertices or both inner vertices.
	If $a$ and $b$ are adjacent, then $a\wedge b$ is the edge joining them, i.e., $a\wedge b=\{a,b\}\in X_2$ and $\delta(a,b) = 1$.
	If $a$ and $b$ are not adjacent, then $a\wedge b=\bot$ and $\delta(a,b) = 3$.
	It remains to consider the mixed case, where, without loss of generality, $a$ is an outer vertex and $b$ is an inner vertex.
	If $a$ and $b$ are joined by a spoke, then again $a\wedge b=\bot$ and $\delta(a,b) = 3$.
	Otherwise, a short check shows that there is a unique opposite pair $x\in X_1$ consisting of an outer edge containing $a$ and an inner edge containing $b$.
	Hence $a\wedge b=x\in X_1$ and $\delta(a,b) = 2$.
	In the precise representation, assume that $a = v^{(o)}_i$ and $b = v^{(i)}_j$.
	In the non-spoke case, the sets $\{i-2,i+2\}$ and $\{j-1,j+1\}$ intersect in a single element $k$, and then $a \wedge b = \{e^{(o)}_k, e^{(i)}_k\}$.

	Now let $a,b,c\in\Omega$ be three consecutive outer vertices.
	Then $\delta(a,b) = \delta(b,c) = 1$ and $\delta(a,c) = 3$.
	Thus, $\delta(a,c) \nleq \delta(a,b) + \delta(b,c)$, so the triangle inequality fails.
	Therefore, $\delta$ is not a metric on $\Omega$.

	We remark that $(X,{\leq})$ is not Delsarte-regular: property~\ref{regular:pi} fails for $\pi(0,1,2)$, $\pi(0,2,2)$, and $\pi(0,2,1)$.
\end{example}

\begin{theorem}[Singleton bound]\label{thm:singleton}
	Let $(X,{\leq})$ be a finite regular poset of height $n$.
        Let $A \subseteq \Omega$ be a $d$-code with $d\in\{1,\ldots,n\}$.
        Then
        \[
                \#A \leq \frac{\theta(0)}{\theta(n-d+1)}\text{.}
        \]
        Equality holds if and only if $A$ is a Steiner system of strength $n-d+1$.
\end{theorem}

\begin{proof}
	Let $t = n-d+1\in\{1,\ldots,n\}$.
	We count the set $S$ of all pairs $(x,y) \in X_t \times A$ with $x\leq y$ in two ways.
	On the one hand, there are $\#A$ elements $y\in A$, and for each such $y$ there are $\nu(t,n)$ choices for $x\in X_t$ with $x \leq y$.
	So $\#S = \#A\cdot \nu(t,n)$.
	On the other hand, there are $\#X_t$ choices for $x$, and since $A$ is a $d$-code, for each such $x$ there is at most one $y\in A$ with $x\leq y$.
	So $\#S \leq \#X_t$.
	Hence,
	\[
		\#A \leq \frac{1}{\nu(t,n)} \cdot \#X_t
		\overset{\text{Lem.~\ref{lem:level_sizes}}}{=} \frac{1}{\nu(t,n)} \cdot \frac{\theta(0)}{\theta(t)}\, \nu(t,n)
		= \frac{\theta(0)}{\theta(n-d+1)}\text{.}
	\]

	Equality holds if and only if every $x\in X_t$ lies below exactly one element of $A$, i.\,e., if and only if $\lambda(x)=1$ for all $x\in X_t$.
        This is precisely the condition that $A$ is a Steiner system of strength $t=n-d+1$.
\end{proof}

As in Mendelsohn's original article~\cite{Mendelsohn-1971}, we now focus on the case $a\in\Omega$.

\begin{theorem}\label{thm:general:steiner_block_intersection_distribution}
	Let $(X,{\leq})$ be a finite regular $\wedge$-semilattice of height $n$.
	Let $D \subseteq \Omega$ be a Steiner system of strength $t$.
	For each block $B\in D$, the \emph{block intersection distribution} at $B$ is given by
	\[
		\alpha_r(B) = \begin{cases}
			\displaystyle \sum_{s=r}^{t-1} \nu'(r,s)\nu(s,n) \left(\frac{\theta(s)}{\theta(t)}-1\right) & \text{if } r\in\{0,\ldots,t-1\}\text{,} \\
			0 & \text{if } r \in \{t,\ldots,n-1\}\text{,} \\
			1 & \text{if } r = n\text{.} \\
		\end{cases}
	\]
	In particular, the block intersection distribution is uniquely determined by the strength $t$ and the regularity parameters of the ambient poset.
\end{theorem}

\begin{proof}
	Let $B'\in D$.
	If $h(B \wedge B') \geq t$, then there exists a $T\in X_t$ with $T \leq B\wedge B'$, that is, with $T \leq B$ and $T \leq B'$.
	Since $D$ is a Steiner system of strength $t$, this forces $B = B'$.
	Hence, $\alpha_n(B) = 1$ and $\alpha_r(B) = 0$ for all $r\in\{t,\ldots,n-1\}$.

	Now for $r\in\{0,\ldots,t-1\}$, the Köhler parameterization from Theorem~\ref{thm:general:mendelsohn}\ref{thm:general:mendelsohn:koehler} yields
	\begin{align*}
		\alpha_r(B)
		& = \sum_{s=r}^t \nu'(r,s)\nu(s,n)\lambda_s - \sum_{u=t+1}^n \Big(\sum_{s=r}^t \nu'(r,s)\nu(s,u)\Big)\delta_{un} \\
		& = \sum_{s=r}^t \nu'(r,s)\nu(s,n)(\lambda_s - 1)\text{.}
	\end{align*}
	Using $\lambda_t = \lambda = 1$ and hence, by Fact~\ref{fct:general:lambda_i}, $\lambda_s = \frac{\theta(s)}{\theta(t)}$, we obtain the claimed expression.
\end{proof}

\begin{remark}
	In coding-theoretic language, Theorem~\ref{thm:general:steiner_block_intersection_distribution} applies to a code $C$ attaining equality in the Singleton bound and determines the distance distribution seen from any individual codeword.
	More precisely, for every codeword $c\in C$, the local distance distribution
	\[
		a_i(c)=\#\{c'\in C\mid \delta(c,c')=i\},
		\qquad (i\in\{0,\ldots,n\})
	\]
	is uniquely determined by the parameters of $C$, and the theorem gives an explicit formula for these numbers.

	It is worth emphasizing that Theorem~\ref{thm:general:steiner_block_intersection_distribution} holds in considerable generality.
	It requires neither an underlying association scheme nor a notion of duality, and it applies to unrestricted codes:
	No linearity assumption, or any comparable algebraic hypothesis, is imposed.
	Moreover, it determines the distance distribution from each individual codeword, not merely an averaged distribution such as Delsarte's inner distribution.

	The proof, based on Theorem~\ref{thm:general:mendelsohn}, isolates the essential double-counting argument.
	The intermediate expressions that occur in the proof are no more involved than the final formula itself, in contrast to approaches based on association schemes and their eigenvalues.
\end{remark}

\begin{example}[Codes and Steiner systems in triangular semilattices]\label{ex:codes_triangular}\thmheadernewline
	Codes in Johnson and $q$-Johnson semilattices are known as constant-weight codes and constant-dimension subspace codes, with distances scaled by the factor $2$.
	The bound from Theorem~\ref{thm:singleton} is found in~\cite[Sec.~4.3.2]{Delsarte-1973-PhilRRSuppl10} for Johnson semilattices and, essentially, in \cite[Th.~9]{Koetter-Kschischang-2008-IEEETIT54[8]:3579-3591} for the $q$-Johnson semilattices.
	Both sources call it the \emph{Singleton bound}.

	Steiner systems are the same as diameter-perfect constant-weight and constant-dimension codes, respectively \cite{Ahlswede-Aydinian-Khachatrian-2001-DCC22[3]:221-237,Schwartz-Etzion-2002-JCTSA97[1]:27-42}.
	It is known that for any strength $t\in\N$, there exist infinitely many Johnson and $q$-Johnson semilattices admitting a non-trivial Steiner system of strength $t$; see \cite{Keevash-2026-AMSS:to_appear,Glock-Kuehn-Lo-Osthus-2023-MemAmerMathSoc284[1406]} and \cite{Keevash-Sah-Sawhney-2025-PLMSTS131[1]:P70071}, respectively.
	However, concrete constructions are notoriously hard to find.
	In Johnson semilattices, only finitely many explicit constructions of Steiner systems with $t\in\{4,5\}$ are known to date, and none are known with $t \geq 6$; as a classical example, we mention the small and the large Witt designs, whose parameters are $5$-$(12,6,1)$ and $5$-$(24,8,1)$ \cite{Witt-1937-AmSUHamb12[1]:265-275}.
	For $q$-Johnson semilattices, examples are even scarcer: the only parameter set for which concrete Steiner systems of strength $t\geq 2$ are known is $2$-$(13,3,1)_2$ \cite{Braun-Etzion-Ostergard-Vardy-Wassermann-2016-ForumMathPi:e7}.

	Theorem~\ref{thm:general:steiner_block_intersection_distribution} yields the block intersection distribution of Steiner and $q$-Steiner systems, equivalently, the distance distribution of the corresponding diameter-perfect constant-weight and constant-dimension codes with respect to a fixed codeword; this is stated below as Corollary~\ref{cor:johnson_steiner_intersection_distribution}.
	For combinatorial designs, the uniqueness of the block intersection distribution is found in~\cite{Mendelsohn-1970-CanadJM22[5]:1010-1015}, and corresponding formula in~\cite{Goethals-1970-InstStatistMimeoSer600_29}.
	For subspace designs, to the best of our knowledge, both the uniqueness statement and the formula are new.
\end{example}

\begin{corollary}\label{cor:johnson_steiner_intersection_distribution}
	Let $D$ be an ordinary or a $q$-analog Steiner system.
	We denote its parameters by $t$-$(v,k,1)_q$, where possibly $q=1$.
	For each block $B\in D$, the \emph{block intersection distribution} at $B$ is given by
	\[
		\alpha_i(B) = \begin{cases}
			\displaystyle\qbinom{k}{i}{q} \sum_{j=0}^{t-i-1} (-1)^j q^{\binom{j}{2}} \qbinom{k-i}{j}{q} \left(\frac{\qbinom{v-i-j}{v-k}{q}}{\qbinom{v-t}{v-k}{q}} - 1\right) & \text{if } i\in\{0,\ldots,t-1\}\text{,} \\
			0 & \text{if } i \in \{t,\ldots,k-1\}\text{,} \\
			1 & \text{if } i = k\text{.} \\
		\end{cases}
	\]
	In particular, the block intersection distribution is uniquely determined by the parameters $v$, $k$, $t$, and, in the $q$-analog case, $q$.
\end{corollary}

\begin{proof}
	We apply Theorem~\ref{thm:general:steiner_block_intersection_distribution}.
	The cases $i\geq t$ are immediate.
	For the remaining cases $i\in\{0,\ldots,t-1\}$, we proceed similarly to the proof of Corollary~\ref{cor:hamming:mendelsohn}\ref{cor:hamming:mendelsohn:koehler}:
	Starting with the formula in Theorem~\ref{thm:general:steiner_block_intersection_distribution}, we specialize the symbols $\theta$, $\mu$, $\nu$, and $\nu'$ to their expressions in the $q$-Johnson scheme and obtain
	\begin{align*}
		\alpha_i(B)
		& = \sum_{s=i}^{t-1} (-1)^{s-i} q^{\binom{s-i}{2}} \qbinom{s}{i}{q} \qbinom{k}{s}{q} \left(\frac{\qbinom{v-s}{k-s}{q}}{\qbinom{v-t}{k-t}{q}} - 1\right) \\
		& = \qbinom{k}{i}{q}\sum_{s=i}^{t-1} (-1)^{s-i} q^{\binom{s-i}{2}} \qbinom{k-i}{s-i}{q} \left(\frac{\qbinom{v-s}{v-k}{q}}{\qbinom{v-t}{v-k}{q}} - 1\right) \\
		& = \qbinom{k}{i}{q}\sum_{j=0}^{t-i-1} (-1)^{j} q^{\binom{j}{2}} \qbinom{k-i}{j}{q} \left(\frac{\qbinom{v-i-j}{v-k}{q}}{\qbinom{v-t}{v-k}{q}} - 1\right)\text{,}
	\end{align*}
	where we used the trinomial identity in the second step.
\end{proof}

\begin{example}
	As an illustration of Corollary~\ref{cor:johnson_steiner_intersection_distribution}, we compute the intersection distribution of a $5$-$(28,7,1)$ Steiner system as
	\[
		(\alpha_0, \alpha_1, \alpha_2, \alpha_3, \alpha_4, \alpha_5, \alpha_6, \alpha_7)
		= (465,\, 1470,\, 1764,\, 735,\, 245,\, 0,\, 0,\, 1)
	\]
	and that of a $2$-$(13,3,1)_2$ binary $q$-Steiner system as
	\[
		(\alpha_0, \alpha_1, \alpha_2, \alpha_3)
		= (\num{1587696},\, \num{9548},\, 0,\, 1)\text{.}
	\]
	Steiner systems with both parameters exist \cite{Denniston-1976-BLMS8[3]:263-267,Braun-Etzion-Ostergard-Vardy-Wassermann-2016-ForumMathPi:e7}.
\end{example}

\begin{example}[Hamming semilattices: Steiner systems and MDS codes]\label{ex:codes_hamming}\thmheadernewline
	Codes in the Hamming semilattices $H(m,n)$ are precisely $m$-ary block codes, and Theorem~\ref{thm:singleton} reproduces the classical Singleton bound \cite{Singleton-1964-IEEETIT10[2]:116-118}.
	Steiner systems of strength $t$ in $H(m,n)$ were originally studied as orthogonal arrays of \emph{index unity} \cite{Bush-1952-AnnMathStatistics23[3]:426-434}.
	Through the coding correspondence, they are the same as $m$-ary MDS codes of length $n$, size $m^t$, and minimum Hamming distance $n-t+1$, that is, MDS codes with parameters $(n,m^t,n-t+1)_m$; see Delsarte~\cite[Sec.~4.3.2]{Delsarte-1973-PhilRRSuppl10} (based on~\cite[Th.~4.4]{Delsarte-1973-PhilRRSuppl10}) or \cite[Th.~4.21]{Hedayat-Sloane-Stufken-1999-OrthogonalArrays} for a different proof.

	For alphabets of the form $M = \F_m$ with a prime power $m$, the (extended) Reed-Solomon codes provide linear $(n,m^t,n-t+1)_m$ MDS codes for all $n\in \{0,\ldots,m+1\}$ and $t\in\{0,\ldots,n\}$.
	In the orthogonal array setting, this construction already appears in~\cite{Bush-1952-AnnMathStatistics23[3]:426-434};
	its standard coding-theoretic name comes from the later independent construction by Reed and Solomon~\cite{Reed-Solomon-1960-JSIAM8[2]:300-304}.
	In the case where $m$ is a power of two and $t\in\{3,m-1\}$, they can be extended further to a linear $(m+2,m^t,m+3-t)_m$ MDS code, essentially corresponding to hyperovals in the projective plane $\PG(2,m)$ under the geometric correspondence described in~\cite{Dodunekov-Simonis-1998-ElecJComb5:R37}.
	The famous MDS conjecture, originating in Segre's 1955 article~\cite{Segre-1955-AnnDMPA39:357-379}, states that the length $n$ of a linear MDS code is at most $m+1$, apart from the trivial cases $t\in\{0,1,n-1,n\}$, where the length is unrestricted, and the two above-mentioned exceptions in even characteristic, where length $m+2$ can occur.
	Dropping the assumptions that the codes are linear and that $m$ is a prime power yields an extension of the MDS conjecture to unrestricted codes.%
	\footnote{
		For non-prime powers $m$, not all parameters of this form need to be realizable.
		It is known that any $(n,m^2,n-1)_m$ MDS code corresponds to a set of $n-2$ pairwise orthogonal Latin squares of size $m\times m$.
		Thus, for example, by the non-existence of a pair of orthogonal Latin squares of order $6$, which is precisely the insolubility of Euler's thirty-six officers problem, there are no senary $(n,36,n-1)_6$ MDS codes with $n \geq 4$.
		Moreover, the realizability question for the parameters of unrestricted MDS codes includes the prime power conjecture for the orders of finite projective planes.
	}
	The MDS conjecture has been established in certain subcases, but remains open in general.
	Significant progress was made in \cite{Ball-2012-JEurMathSocJEMS14[3]:733-748}, where it was proven for linear codes over all prime fields $\F_m$.

	Theorem~\ref{thm:general:steiner_block_intersection_distribution} recovers the local distance distribution of MDS codes.
	In the special case of linear MDS codes over finite fields $M = \F_m$, where the local distance distribution equals the weight distribution, this result was first obtained independently in~\cite{Assmus-Mattson-Turyn-1965-AFCRL67_332, Kasami-Lin-Peterson-1966-IEEETIT12[2]:274, Forney-1966-MITPressResMonogr37}.
	Still formulated as the weight distribution of a linear MDS code, the core double-counting argument that also works for unrestricted codes seems to have appeared first in~\cite[Th.~6]{Goethals-1969-PhilRR24:145-159}.
	Based on this argument, essentially the same as the one used in our proof, a formulation explicitly covering unrestricted codes over arbitrary alphabets can be found in~\cite[p.~207]{Heise-Quattrocchi-1983-Codierungstheorie_1st}.

	The classic monograph of MacWilliams and Sloane~\cite{MacWilliams-Sloane-1977-The_Theory_of_Error_Correcting_Codes_I} devotes its Chapter~11 to linear MDS codes, describing the subject as \enquote{one of the most fascinating chapters in all of coding theory} and the uniqueness of their weight distribution is called \enquote{surprising}.
	A further, more recent standard source on MDS codes is \cite[Ch.~6]{Ball-2020-ACourseInAlgebraicErrorCorrectingCodes}.
\end{example}

\begin{example}[$q$-Hamming semilattices: Steiner systems and MRD codes]\label{ex:codes_q_hamming}\thmheadernewline
	Codes in the $q$-Hamming semilattice $H_q(m,n)$ are $m\times n$ rank-metric codes over $\F_q$.
	Surveys are found in~\cite{Gorla-Ravagnani-2018-SignalsCommunTechnol:3-23} and~\cite{Bartz-Holzbaur-Liu-Puchinger-Renner-WachterZeh-2022-FoundTrendsCommunInfTheory19[3]:390-546}.
	Theorem~\ref{thm:singleton}, including the characterization of the equality cases, is found under the name \emph{Singleton bound} in \cite[p.~227]{Delsarte-1978-JCTSA25[3]:226-241}.
	Steiner systems of strength $t$ in $H_q(m,n)$ correspond to $m \times n$ MRD codes over $\F_q$ of size $q^{mt}$ and minimum rank distance $n-t+1$.
	Remarkably, and in sharp contrast to the situation in the ordinary Hamming semilattice, they exist for all parameters for which the semilattice $H_q(m,n)$ is short, i.\,e., for $m \geq n$.
	The first construction of MRD codes appeared as \emph{Singleton systems} in \cite{Delsarte-1978-JCTSA25[3]:226-241} and was later rediscovered in \cite{Gabidulin-1985-ProblemsInformTransm21[1]:1-12}, which led to the now-standard name \emph{Gabidulin codes}, and again in \cite{Roth-1991-IEEETIT37[2]:328-336}.
	The construction yields $\F_q$-linear MRD codes and may be viewed as a $q$-analog of the Reed-Solomon construction.

	Theorem~\ref{thm:general:steiner_block_intersection_distribution} recovers the local distance distribution of MRD codes.
	The corresponding statement for the inner distribution, which is essentially the local distance distribution averaged over all codewords, was proven by Delsarte in~\cite[Th.~5.6]{Delsarte-1978-JCTSA25[3]:226-241} using the theory of association schemes; the proof involves rather complicated eigenvalue expressions for the $q$-Hamming scheme, namely the corresponding generalized Krawtchouk polynomials.
	By the translation-invariance of the rank metric, linear MRD codes $C$ are \emph{distance-homogeneous}, meaning that the local distance distribution at $c$ does not depend on the chosen codeword $c\in C$.
	It follows that this common local distance distribution agrees with the inner distribution and, by linearity, also with the rank distribution of $C$.
	Therefore, Delsarte's result implies the general result for $\F_q$-linear MRD codes, which was also found in~\cite[Th.~5]{Gabidulin-1985-ProblemsInformTransm21[1]:1-12} and~\cite[Th.~3]{Dumas-Gow-McGuire-Sheekey-2010-LAA433[1]:191-202} by a combination of linear algebra and counting arguments, and in~\cite[Cor.~44 and Rem.~46]{Ravagnani-2016-DCC80[1]:197-216}, using duality and MacWilliams identities for linear rank-metric codes.%
	\footnote{
		Note that the local distance distribution is not stated in closed form in~\cite{Ravagnani-2016-DCC80[1]:197-216}.
		Rather,~\cite[Rem.~46]{Ravagnani-2016-DCC80[1]:197-216} provides a recursive formula for the rank distribution of a nonzero linear MRD code.
	}
	The latter approach was subsequently extended to a more general setting and to unrestricted codes in~\cite{Ravagnani-2018-DCC86[9]:2035-2063}.
	Corollary~51 of that source implies, in particular, that unrestricted MRD codes, and also MDS codes in the corresponding group setting, are distance-homogeneous; however, no explicit closed formula for the local distance distribution is included.
	Taken together with the explicit formula from~\cite[Th.~5.6]{Delsarte-1978-JCTSA25[3]:226-241}, this yields the full content of Theorem~\ref{thm:general:steiner_block_intersection_distribution} for MRD codes and appears to be the first published route to the result in that generality.
	However, we are not aware of a source that applies the combinatorial double-counting argument directly to obtain this formula for unrestricted MRD codes.
\end{example}

\section{Application to perfect matchings}\label{sec:matchings}

In this section, we apply our theory to the perfect matchings of a complete graph.
Let $n$ be a nonnegative integer and let $K_{2n}$ be the complete graph on $2n$ vertices.
A \emph{matching} is a set of pairwise disjoint edges.
A matching $M$ of $K_{2n}$ is called \emph{perfect} if $\#M = n$.
Let $X$ be the set of all matchings of $K_{2n}$.
Then $(X,{\subseteq})$ is a finite $\cap$-semilattice, satisfying the regularity properties \ref{regular:theta}, \ref{regular:mu}, and \ref{regular:nu} with
\begin{align*}
        \theta(r,s) &= \binom{2(n-r)}{2(s-r)}(2(s-r)-1)!!\text{,} &
        \mu(r,s) &= \binom{n-r}{s-r}\text{,}
        &\nu(r,s) &= \binom{s}{r}\text{,}
\end{align*}
where for $k\in\N$,
\[
        (2k-1)!! = \begin{cases}
		1\cdot 3 \cdot 5 \cdot \ldots \cdot (2k-1) & \text{if }k \geq 1\text{,} \\
		1 & \text{if }k = 0
	\end{cases}
\]
is the double factorial, counting the number of perfect matchings of a complete graph on $2k$ vertices.
The semilattice $(X,{\subseteq})$ is graded with height function $h(M) = \#M$.
However, $(X,{\subseteq})$ is not short, since a matching with $n-1$ edges has a unique completion to a perfect matching and hence, cannot be represented as the intersection of two perfect matchings.

One can also study all matchings of $K_{2n}$ with at most $k$ edges for some nonnegative integer $k \leq n$.
If $k < n$, then $\mu(r,s)=\binom{k-r}{s-r}$, while the formulas for $\nu(r,s)$ and $\theta(r,s)$ remain unchanged.
Here, we focus on the extremal case $k = n$.

We emphasize that for $n \geq 4$, the top level $\Omega$ of all perfect matchings does not give rise to an association scheme from these meet-height relations.
For an explicit counterexample, notice that the number $\pi(0,2,2)$ is not well-defined.
Take the matchings $a=\{13,24\}$ and $b = \{12,34,56,78\}$ of $K_8$.
Then $a \wedge b = \emptyset$.
The only perfect matching that contains~$a$ and shares precisely $2$ edges with $b$ is $\{13,24,56,78\}$, so the corresponding count is $1$.
However, if $\widetilde{a} = \{23,45\}$, then again $\widetilde{a} \wedge b = \emptyset$, but now there is no perfect matching containing $\widetilde{a}$ that shares precisely $2$ edges with $b$, so the corresponding count is $0$.
Clearly, the same construction extends to any $n\geq 4$: add the same $n-4$ disjoint edges to both $a$ and $b$, so that $h(a\wedge b)=n-4$, and compare the corresponding values of $\pi(n-4,n-2,n-2)$.

We point out that one can define an association scheme on the perfect matchings of~$K_{2n}$ by defining finer relations than the ones coming from the height function $h$; see~\cite{Bamberg-Klawuhn-2026-AlgebrComb9[3]:789-809} for details.
However, the semilattice is regular in the sense of this article, and hence we can study codes and designs of perfect matchings without the structure of an association scheme.

A $t$-design $D\subseteq\Omega$ is precisely an $(n-t,1,\ldots,1)$-factorization as studied in~\cite{Bamberg-Klawuhn-2026-AlgebrComb9[3]:789-809}.
That is, $D$ is a subset of the perfect matchings of~$K_{2n}$ with the property that every set of $t$ disjoint edges of $K_{2n}$ is contained in a constant number of elements of $D$.
By Fact~\ref{fct:general:lambda_i}, every $t$-design $D$ of strength $t$ and index $\lambda$ is also an $s$-design for all $s\in\{0,\ldots,t\}$, and the associated index is
\[
	\lambda_s = \frac{(2(n-s) - 1)!!}{(2(n-t)-1)!!} \cdot \lambda\text{.}
\]
The divisibility conditions from Theorem~\ref{thm:divisibility}\ref{thm:divisibility:combined} reduce to $1 \mid \lambda$, so the integrality conditions for $t$-designs of perfect matchings are trivial.

The existence problem for $t$-designs of perfect matchings remains wide open.
Cameron constructed $2$-designs of index~$1$, i.\,e., $2$-Steiner systems, using hyperovals in finite projective planes~\cite[Th.~7.3.(i)]{Cameron-1976-ParallelismsCompleteDesigns}.
No constructions of nontrivial $t$-designs with $t \geq 3$ are currently known.

Theorem~\ref{thm:general:mendelsohn} yields the following Mendelsohn equations for $t$-designs of perfect matchings.

\begin{corollary}\label{cor:matchings:mendelsohn}
Let $n$ be a nonnegative integer and let $\Omega$ be the set of perfect matchings of the complete graph $K_{2n}$.
Let $D \subseteq \Omega$ be a $t$-design of index $\lambda$ in the perfect matching semilattice and let $M$ be a (not necessarily perfect) matching of $K_{2n}$.
Then the following systems of equations hold, and they are equivalent.
	\begin{enumerate}[(a)]
		\item\label{cor:matchings:mendelsohn:mendelsohn}\emph{Mendelsohn equations.}\nobreak\hspace{1em}\ignorespaces
		For all $r\in\{0,\ldots,t\}$,
		\[
			\sum_{s=r}^n \binom{s}{r} \alpha_s(M) = \binom{\#M}{r} \lambda_r\text{.}
		\]
		The right-hand side admits the alternative expressions
		\[
			\binom{\#M}{r} \lambda_r
			= \binom{\#M}{r} \frac{(2(n-r)-1)!!}{(2(n-t)-1)!!} \lambda
			= \binom{\#M}{r} \frac{(2(n-r)-1)!!}{(2n-1)!!} \#D\text{.}
		\]
		\item\label{cor:matchings:mendelsohn:koehler}\emph{Köhler parametrization.}\nobreak\hspace{1em}\ignorespaces
		For all $r\in\{0,\ldots,t\}$,
		\begin{align*}
			\alpha_r(M)
			& = \binom{\#M}{r} \sum_{s=0}^{t-r} (-1)^s\binom{\#M-r}{s} \lambda_{r+s} \\
			& \phantom{~=~} - \sum_{u=t+1}^n \binom{u}{r}\left(\sum_{s=0}^{t-r} (-1)^{s}\binom{u-r}{s}\right)\alpha_u(M)\text{,}
		\end{align*}
		where the inner sum on the right-hand side admits the alternative form
		\[
			\sum_{s=0}^{t-r} (-1)^s\binom{u-r}{s} = -\sum_{s=t+1-r}^{u-r} (-1)^{s}\binom{u-r}{s}\text{.}
		\]
	\end{enumerate}
\end{corollary}

\begin{proof}
	We apply the same strategy as in the proof of Corollary~\ref{cor:hamming:mendelsohn}.
\end{proof}

Theorem~\ref{thm:general:steiner_block_intersection_distribution} shows that the block intersection distribution of a Steiner system is uniquely determined by its parameters.
By similar transformations as in the proof of Corollary~\ref{cor:matchings:mendelsohn}\ref{cor:matchings:mendelsohn:koehler}, we get the following explicit formula.

\begin{corollary}\label{cor:matchings:steiner_block_intersection_distribution}
	Let $n$ be a nonnegative integer and let $\Omega$ be the set of perfect matchings of the complete graph $K_{2n}$.
	Let $D \subseteq \Omega$ be a Steiner system of strength $t$, i.\,e., a $t$-design of index $1$.
	For each perfect matching $M \in D$, the intersection numbers at $M$ are given by
	\[
		\alpha_r(M) = \begin{cases}
			\displaystyle
			\binom{n}{r} \sum_{s=0}^{t-1-r} (-1)^s \binom{n-r}{s}\!\! \left( \frac{(2(n-r-s)-1)!!}{(2(n-t)-1)!!} -1 \right) & \text{if } r\in\{0,\ldots,t-1\}\text{,} \\[1mm]
			0 & \text{if } r \in \{t,\ldots,n-1\}\text{,} \\
			1 & \text{if } r = n\text{.} \\
		\end{cases}
	\]
\end{corollary}

\begin{remark}
	For ordinary designs and orthogonal arrays, the formula corresponding to Corollary~\ref{cor:matchings:steiner_block_intersection_distribution} can be used for non-existence results.
	However, evaluating Corollary~\ref{cor:matchings:steiner_block_intersection_distribution} for small values of $n$ and $t$ always yields non-negative numbers and, therefore, does not give a contradiction.
	In particular, the non-existence of a Steiner system of strength~$3$ in $K_{12}$ that was shown in~\cite[Sec.~5]{Bamberg-Klawuhn-2026-AlgebrComb9[3]:789-809} does not follow from inspecting the intersection numbers.
\end{remark}

Finally, we point out that our definition of a $d$-code coincides with the notion of a $d$-code introduced in \cite[Sec.~6.5]{Klawuhn-2026-UPB}.
Here, the distance of two perfect matchings is measured in terms of the size of their intersection.
The Singleton bound in Theorem~\ref{thm:singleton} specializes to \cite[Th.~6.5.3]{Klawuhn-2026-UPB}.
This shows that the concept of $d$-codes in semilattices is sensible, even in the absence of the structure of an association scheme.

\section*{Acknowledgements}

This project emerged during a research visit of the second author to the University of Bayreuth.
He would like to thank the Deutsche Forschungsgemeinschaft (DFG, German Research Foundation) -- Project-ID 491392403 -- TRR 358 for funding the research visit.

\printbibliography
\end{document}